\documentclass[reqno,12pt]{amsart}
\usepackage{a4wide,color,eucal,enumerate,mathrsfs}
\usepackage[normalem]{ulem}
\usepackage{amsmath,amssymb,epsfig,amsthm} 
\usepackage[latin1]{inputenc}
\usepackage{psfrag}
\usepackage{hyperref,cleveref,tikz}

\def\nn{{\mathbb N}}

\def\ct{{\mathcal T}}

\def\az{\alpha}

\def\dz{\delta}

\def\bz{\beta}

\def\gz{{\gamma}}

\def\sz{\sigma}

\def\bint{{\ifinner\rlap{\bf\kern.35em--}
\int\else\rlap{\bf\kern.45em--}\int\fi}\ignorespaces}

\def\bbint{{\ifinner\rlap{\bf\kern.35em--}
\hspace{0.078cm}\int\else\rlap{\bf\kern.45em--}\int\fi}\ignorespaces}

\newtheorem{thm}{Theorem}[section]
\newtheorem{lem}[thm]{Lemma}
\newtheorem{prop}[thm]{Proposition}

\numberwithin{equation}{section}

\theoremstyle{remark}
\newtheorem{rem}[thm]{Remark}

\def\bint{{\ifinner\rlap{\bf\kern.35em--}
\int\else\rlap{\bf\kern.45em--}\int\fi}\ignorespaces}

\usepackage{graphicx}
\usepackage{import}
\usepackage{xifthen}
\usepackage{pdfpages}
\usepackage{transparent}

\title[Weak Bernoulli laws without Alt--Caffarelli regularity]{Serrin's overdetermined theorem and Weak Bernoulli laws without Alt--Caffarelli regularity}
\author{Yi Ru-Ya Zhang}
\date{\today}

\address{State Key Laboratory of Mathematical Sciences, Academy of Mathematics and Systems Science, Chinese Academy of Sciences, Beijing 100190, China}
\address{Academy of Mathematics and Systems Science, the Chinese Academy of Sciences, Beijing 100190, China}
\email{yzhang@amss.ac.cn}
 \thanks{The author is funded by the National Key R\&D Program of China (Grant No. 2025YFA1018400 \&  No. 2021YFA1003100), NSFC Grant No. 12288201 \& No. 12571128, the Chinese Academy of Sciences, and CAS Project for Young Scientists in Basic Research, Grant No. YSBR-031. }

\subjclass[2020]{30C35, 35N25, 35R35}
\keywords{Weak Bernoulli law, Serrin's overdetermined problem, Smirnov domain}

\begin{document}

\begin{abstract}
We study distributional Bernoulli-type conditions in geometrically irregular domains $\Omega$. Here the zero extension of $u$ to $\mathbb{R}^n$ satisfies
$$
\Delta u \;=\; c\,\mathcal{H}^{n-1}\!\lfloor_{\partial^*\Omega}\;-\;f(u)\,\mathbf{1}_\Omega\,dx
$$
in the distributional sense. This is a weak version of the one-phase Bernoulli free boundary condition, which avoids the uniform Lipschitz/density assumptions of classical Alt--Caffarelli theory.

We prove that for every $n \ge 2$ and every $f \in C^2(\mathbb{R})$ with $f(0)>0$, there exist bounded, non-spherical, finite-perimeter domains $\Omega\subset \mathbb R^n$ satisfying this distributional Bernoulli law with
$$
 0<\mathcal H^{n-1}(\partial^*\Omega)<\infty,
 \qquad \mathcal H^{n-1}(\partial\Omega\setminus\partial^*\Omega)=0,
$$
yet
$$
 {\rm ess}\sup_{x\in\partial^*\Omega}
 \sup_{0<r<1} \frac{\mathcal H^{n-1}(B_r(x)\cap\partial^*\Omega)}{r^{n-1}} =\infty.
$$ 
This shows the key constraint is not absence of a reduced boundary, but failure of uniform all-scale surface density bounds. For $f \equiv 1$, these results yield counterexamples to the weak Serrin-type overdetermined problems in all dimensions, proving the distributional Bernoulli law alone cannot replace the uniform growth/density conditions core to Alt--Caffarelli theory.

On the other hand, we prove a planar rigidity result: Within the Smirnov class, the associated harmonic quadrature identity forces $\Omega$ to be a disk. Thus, for the constant-source Serrin/Bernoulli law, Smirnov regularity is the threshold for weak Bernoulli rigidity in $\mathbb R^2$, while uniform upper density bounds form the threshold for $n\ge 3$ according to \cite{FZ2025}.
\end{abstract}


\maketitle

\section{Introduction}

A classical theorem of Serrin states that a bounded $C^2$ domain
$\Omega\subset\mathbb R^n$ which supports a solution of the overdetermined
boundary value problem
\begin{equation}\label{eq:classical-serrin}
    -\Delta u = 1 \quad\text{in }\Omega,\qquad
    u=0 \quad\text{on }\partial\Omega,\qquad
    \partial_{\nu_{\rm in}}u=c \quad\text{on }\partial\Omega
\end{equation}
must be a ball. Here $\nu_{\rm in}$ denotes the inward unit normal and
$c$ is a constant. Serrin's original proof \cite{S1971} is based on
the moving plane method, while Weinberger later gave a proof using a
$P$-function argument \cite{W1971}. Since then, Serrin-type
overdetermined problems have become a central theme in elliptic PDE,
potential theory, and geometric analysis; see, for example,
\cite{BNST2008,CH1998,PS1989,CWZ2025} and the surveys
\cite{S2001,NT2018}.

Serrin's problem may also be viewed as a one-phase free-boundary problem. If the domain is unknown, then the Dirichlet condition fixes the positivity phase, while the constant normal derivative is the Bernoulli condition on the free boundary.  In the variational one-phase theory of Alt and Caffarelli \cite{AltCaffarelli1981}, this Bernoulli law is accompanied by additional growth and density hypotheses, which are essential for regularity of the
reduced free boundary.  Related regularity theories were developed through Caffarelli's viscosity approach
\cite{Caffarelli1987,Caffarelli1989a,Caffarelli1988}, the Caffarelli--Salsa theory \cite{CaffarelliSalsa2005}, De Silva's flatness
theorem \cite{DeSilva2011}, and the almost-minimizer theory of David--Toro \cite{DavidToro2015} and De Silva--Savin  \cite{DeSilvaSavin2019}.

The purpose of the present paper is to isolate what remains if one keeps only
the distributional Bernoulli identity and discards the accompanying growth or
density assumptions.

Even more generally, let $\Omega\subset\mathbb R^n$ be a bounded set with countably
$(n-1)$-rectifiable boundary and finite perimeter measure.  A compactly
supported function $u$, extended by zero outside $\Omega$, is said here to
satisfy a \emph{distributional Bernoulli law} with source $f(u)$ if
\begin{equation}\label{serrin}
    \Delta u
    =
    c\,\mathcal H^{n-1}\!\lfloor_{\partial^*\Omega}
    - f(u)\mathbf 1_\Omega\,dx
    \qquad\text{in }\mathcal D'(\mathbb R^n),
\end{equation}
where $\partial^*\Omega$ denotes the reduced boundary of $\Omega$ and
$c>0$ is a constant.  In the planar Jordan-domain part of the paper one may
read $\partial^*\Omega$ as $\partial\Omega$ up to $\mathcal H^1$-null sets.
Equivalently,
\begin{equation*}
\int_{\mathbb R^n} u\,\Delta\varphi\,dx
=  c\int_{\partial^*\Omega}\varphi\,d\mathcal H^{n-1}
- \int_\Omega f(u)\varphi\,dx
 \qquad \forall\,\varphi\in C_c^\infty(\mathbb R^n).
\end{equation*}
When $f\equiv1$, this is the weak form of Serrin's overdetermined problem.


It is important to distinguish \eqref{serrin} from the weak solutions in the Alt--Caffarelli one-phase theory. In  \cite[Section 5]{AltCaffarelli1981}, Alt and Caffarelli imposed, in addition to the basic properties of $u$ and the  measure law, a uniform two-sided growth condition at free-boundary points.
Their density criterion shows that this growth condition is equivalent, for bounded Bernoulli coefficients, to two-sided $(n-1)$-dimensional density bounds for the free-boundary measure.  The examples constructed below satisfy the distributional Bernoulli law \eqref{serrin}, but fail the corresponding uniform upper density condition.  Hence they are not Alt--Caffarelli weak solutions in the full sense.  Rather, they show that the upper growth (or density)
hypotheses in the Alt--Caffarelli weak theory are not technical conveniences: They exclude precisely the singular boundary factors used in our construction.

The importance of the missing density hypothesis is also illustrated by the rigidity theorem of \cite{FZ2025}. Namely, if $f\equiv1$ and $\Omega$ is a bounded domain whose
boundary satisfies the uniform upper density condition
\begin{equation}\label{upper density}
    \mathcal H^{n-1}(\partial\Omega\cap B_r(x))
    \leq Lr^{n-1}
    \qquad
    \forall\,x\in\partial\Omega,\quad 0<r<r_0,
\end{equation}
for some $L>0$ and $r_0>0$, then any solution of \eqref{serrin} forces
$\Omega$ to be a ball. Related approaches in Lipschitz domains were developed
in \cite{DZ2025,DR2026}.  These results are in the same spirit as weak forms
of Alexandrov's theorem for sets of finite perimeter \cite{DM2019}.

\subsection{From Bernoulli laws to harmonic quadrature identities}

The distributional Bernoulli law naturally gives a harmonic moment identity.
Suppose first that $f\equiv1$ and that $u$ satisfies \eqref{serrin}.  If
$h$ is harmonic in an open neighborhood of $\overline\Omega$, then testing
\eqref{serrin} against $\eta h$, where $\eta\in C_c^\infty$ equals one near
$\overline\Omega$, gives
\begin{equation}\label{eq:local-quadrature}
    \int_\Omega h\,dx
    =
    c\int_{\partial^*\Omega}h\,d\mathcal H^{n-1}.
\end{equation}
If, in addition, $\Omega$ is regular for the Dirichlet problem, then the
identity extends to all harmonic functions continuous up to the boundary:
\begin{equation}\label{mean value Serrin}
    \int_\Omega h\,dx
    =
    c\int_{\partial\Omega}h\,d\mathcal H^{n-1}
    \qquad
    \forall\,h\in C(\overline\Omega),\quad \Delta h=0\text{ in }\Omega .
\end{equation}
For a general source $f(u)$, the corresponding moment identity is
\begin{equation}\label{eq:semilinear-quadrature-intro}
    \int_\Omega f(u)h\,dx
    =
    c\int_{\partial\Omega}h\,d\mathcal H^{n-1}
    \qquad
    \forall\,h\in C(\overline\Omega),\quad \Delta h=0\text{ in }\Omega .
\end{equation}
Conversely, if $u$ solves the zero Dirichlet problem
$-\Delta u=f(u)$ in $\Omega$ and \eqref{eq:semilinear-quadrature-intro}
holds, then the zero extension of $u$ satisfies \eqref{serrin}.  This simple
observation is made precise in Lemma~\ref{lem:quadrature-to-weak-serrin}.

The identity \eqref{mean value Serrin} is substantially stronger than the
classical one-point mean-value properties. Epstein and Schiffer proved that
if $0\in\Omega$ and
$$
    \frac1{|\Omega|}\int_\Omega h\,dx=h(0)
    \qquad
    \forall\,h\in C(\overline\Omega),\quad \Delta h=0\text{ in }\Omega,
$$
then $\Omega$ is a ball centered at the origin \cite{ES1965}.  On the other
hand, boundary mean-value identities of the form
\begin{equation}\label{pseudoballs}
    a\int_{\partial\Omega}h\,d\mathcal H^{n-1}=h(0)
    \qquad
    \forall\,h\in C(\overline\Omega),\quad \Delta h=0\text{ in }\Omega
\end{equation}
are more flexible. In the plane, Keldysh--Lavrentiev \cite{KL1937} and Duren--Shapiro--Shields \cite{DSS1966} constructed non-ball domains satisfying
\eqref{pseudoballs}; higher-dimensional analogues were later obtained by Lewis--Vogel \cite{LV1991}.


In planar simply connected domains, \eqref{mean value Serrin} belongs to a circle of problems connecting Serrin-type overdetermined boundary conditions,
quadrature domains, arclength quadrature identities, analytic content, and single-layer potential free-boundary problems.  Khavinson \cite{K1987}
studied an equivalent approximation-theoretic formulation under additional regularity assumptions.  The analytic-boundary case and related free-boundary
formulations were developed in subsequent works; see, for instance, \cite{EKS2002,KSV2005,BK2006,KLT2013,ABKT2019,KL2022}, as well as \cite{GK1994,GK1996} for related approximation problems involving harmonic
vector fields.

The present paper firstly treats the rough simply connected case in the plane.  Namely, in the planar rectifiable Jordan class, we show that the Smirnov condition is the exact threshold between rigidity and flexibility.  In higher dimensions, we obtain counterexamples by replacing the conformal construction with a weighted axially symmetric reduction.

\subsection{Main results}

We first state the planar rigidity and counterexample results. Let $\Omega\subset\mathbb C$ be a
bounded Jordan domain with rectifiable boundary, and let
$ \phi:\mathbb D\to\Omega $
be a conformal map from the open unit disk $\mathbb D$ onto $\Omega$. Since
$\partial\Omega$ is rectifiable, one has $\phi'\in H^1(\mathbb D)$. We say that
$\Omega$ is a Smirnov domain if $\phi'$ is an outer function. Equivalently,
\begin{equation}\label{eq:smirnov-condition}
    \log |\phi'(w)|
    =
    \frac1{2\pi}
    \int_0^{2\pi}
    \frac{1-|w|^2}{|e^{it}-w|^2}
    \log |(\phi')^*(e^{it})|\,dt,
    \qquad w\in\mathbb D .
\end{equation}
This condition is independent of the choice of the conformal map.

Our first theorem proves the rigidity side of the planar theory for the constant source.  It strengthens the planar part of \cite{FZ2025} by replacing
the uniform upper density assumption with the weaker analytic Smirnov condition.  A Smirnov domain need not satisfy \eqref{upper density}; conversely, uniform upper density implies the relevant Smirnov-type boundary behavior in this setting; see, for example, \cite[Theorem~7.6 and Exercise~7.3]{P1992}.

\begin{thm}\label{thm:smirnov-case}
    Let $\Omega\subset\mathbb C$ be a bounded Jordan domain with
    rectifiable boundary. Assume that $\Omega$ is a Smirnov domain. If
    there exists $c>0$ such that \eqref{mean value Serrin} holds,
    then $\Omega$ is a disk.
\end{thm}

The proof combines the quadrature identity with Smirnov's theory of analytic
functions. Starting from \eqref{mean value Serrin}, we derive a boundary
Cauchy-orthogonality relation. Within the Smirnov class, this relation yields a
bounded holomorphic function whose boundary trace encodes the overdetermined
normal derivative. This construction produces a weak Serrin solution and then
allows one to apply Weinberger's $P$-function method in a low-regularity planar
setting.

The Smirnov assumption in Theorem~\ref{thm:smirnov-case} is sharp.  The
counterexample below is stated directly in semilinear Bernoulli form; the case
$f\equiv1$ gives the weak Serrin counterexample.

\begin{thm}\label{thm:counterexample-branch}
    Let $f\in C^2((-\rho,\rho))$ for some $\rho>0$, and assume that
    $f(0)>0$.  Then there exist a bounded rectifiable Jordan domain
    $\Omega\subset\mathbb C$, a nonnegative compactly supported function
    $u\in W^{1,2}(\mathbb C)$, and a constant $c>0$ such that $\Omega$ is
    not a disk, $\Omega$ is not a Smirnov domain, and
    \begin{equation}\label{eq:main-semilinear-law}
        \Delta u
        =
        c\,\mathcal H^1\!\lfloor_{\partial\Omega}
        - f(u)\mathbf 1_\Omega\,dx
        \qquad\text{in }\mathcal D'(\mathbb C).
    \end{equation}
    Moreover, $u\equiv0$ on $\mathbb C\setminus\overline\Omega$, and
    $\Omega$ may be chosen invariant under rotation by $\pi/2$. Indeed, for $a_0>0$ small enough, one can even choose a family   $\{\Omega_a\}_{0<a<a_0}$ of bounded rectifiable Jordan domains. 
\end{thm}
\begin{rem}
    Since $\partial \Omega$ is rectifiable, then 
    $$
 0<\mathcal H^{n-1}(\partial^*\Omega)<\infty,
 \qquad \mathcal H^{n-1}(\partial\Omega\setminus\partial^*\Omega)=0,
$$
Moreover, if
$$
 {\rm ess}\sup_{x\in\partial^*\Omega}
 \sup_{0<r<1} \frac{\mathcal H^{1}(B_r(x)\cap\partial^*\Omega)}{r}<\infty.
$$ 
 then the result of  \cite{Z1984} (see also \cite[Theorem 7.6]{P1992}) implies that such a boundary curve bounds a Smirnov domain, which leads to a contradiction. Thus there is no uniform upper bound on the $\mathcal H^1$-density of the boundary. 
\end{rem}


Consequently, the distributional Bernoulli law \eqref{serrin} alone does not imply free-boundary regularity or radial symmetry, even among rectifiable
Jordan domains in the plane.  The obstruction is a nontrivial singular inner factor in the conformal derivative.  This singular factor is invisible in the boundary arclength density but appears in the interior Jacobian density, and therefore it affects the harmonic moment identity.  The construction starts from the
Duren--Shapiro--Shields singular-factor mechanism and then corrects the boundary density by solving a nonlinear fixed-point equation for an outer
factor.  The imposed four-fold symmetry removes the low Fourier modes associated with translations of the disk, so no Lyapunov--Schmidt reduction is needed in the planar construction.

Recall that a planar domain $\Omega$ is called a \emph{quasidisk} if $\Omega=\phi(\mathbb D)$ for a global quasiconformal mapping $\phi:\mathbb C\to\mathbb C$, and its boundary is then called a \emph{quasicircle}.  Here $\phi$ is quasiconformal if
$\phi\in W^{1,1}_{\rm loc}(\mathbb C;\mathbb C)$ is a homeomorphism and there is a constant $K\ge1$ such that
$$
    |D\phi|^2(x)\le KJ_\phi(x)
    \qquad\text{for a.e. }x\in\mathbb C,
$$
where $J_\phi$ is the Jacobian determinant of $\phi$.  Every quasiconformal mapping is H\"older continuous; see \cite[Chapter~13]{AIM2009}.  The examples
in Theorem~\ref{thm:counterexample-branch} may be chosen with $C^{0,\alpha}$ quasicircle boundary for every $0<\alpha<1$.  Thus, within this quasicircle scale, Lipschitz or Ahlfors-regular boundary control is the sharp threshold for the weak Bernoulli rigidity problem.

The same flexibility persists in every higher dimension, while the construction is no longer a formal consequence of the planar theorem. 

\begin{thm}\label{thm:higher-dimensional-counterexample}
Let $n\ge3$.  Let $f\in C^2((-\rho,\rho))$ for some $\rho>0$, and assume that $f(0)>0$.  Then there exist a bounded set of finite perimeter $\Omega\subset\mathbb R^n$, a nonnegative compactly supported function $u\in W^{1,2}(\mathbb R^n)$, and a constant $c>0$ such that $\Omega$ is not a ball and \begin{equation}\label{eq:main-higher-dimensional-law}
 \Delta u =c\,\mathcal H^{n-1}\!\lfloor_{\partial^*\Omega}  -f(u)\mathbf 1_\Omega\,dx \qquad\text{in }\mathcal D'(\mathbb R^n).
 \end{equation}  
 Moreover $u\equiv0$ on $\mathbb R^n\setminus\overline\Omega$.  The set  $\Omega$ may be chosen to be axially symmetric and homeomorphic, as a manifold with boundary, to $\mathbb B^2\times\mathbb S^{n-2}$. 

In addition,
$$
 0<\mathcal H^{n-1}(\partial^*\Omega)<\infty,
 \qquad
 \mathcal H^{n-1}(\partial\Omega\setminus\partial^*\Omega)=0,
$$
and
$$
 {\rm ess}\sup_{x\in\partial^*\Omega}
 \sup_{0<r<1}
 \frac{\mathcal H^{n-1}(B_r(x)\cap\partial^*\Omega)}{r^{n-1}}
 =\infty.
$$
Therefore it gives a counterexample to the rigidity of Bernoulli laws without the upper density estimate \eqref{upper density}, and yields the sharpness of the result \cite{FZ2025} when $n\ge 3$ by taking $f\equiv 1$. 
\end{thm}

We remark that, the last part of  Theorem~\ref{thm:higher-dimensional-counterexample} is  consistent with De Giorgi's blow-up theorem: At every point of the reduced boundary the infinitesimal blow-up is still a half-space with a locally finite upper density at each point in $\partial^*\Omega$, but no uniform all-scale bound independent of the point. Indeed, the uniform density upper bound fails when $0<r<1$ due to an intermediate-scale clustering phenomenon.

Although Theorems~\ref{thm:counterexample-branch} and
\ref{thm:higher-dimensional-counterexample} have the same distributional
Bernoulli conclusion, their proofs are substantially different.  Recall that, in the planar
case the conformal map is used directly.  The disk-side equation associated to the Duren--Shapiro--Shields construction involves the
ordinary Poisson balayage operator $\mathcal T$, and the four-fold symmetry removes the Fourier modes generated by self-M\"obius maps of the disk.  The remaining
linearized operator is
$$
    I-2\mathcal T\circ P,
$$
which is invertible on the four-fold symmetric mean-zero space.  The semilinear term is a lower-order perturbation after the scaling
$u=\lambda^2v$, since
$$
    f(\lambda^2v)=f(0)+O(\lambda^2).
$$

The higher-dimensional construction does not arise as a straightforward rotational extension of the planar one, although certain conceptual features are analogous. In contrast to the planar setting, the higher-dimensional configurations should not be interpreted as a compact perturbative family of non-spherical domains in $\mathbb{R}^n$. In fact, they are small singular perturbations of the disk only in the normalized meridian variables. Moreover, after the necessary  translation to height $R = \tau^{-1}$ and an appropriate axial rotation, the unscaled domains diverge to infinity as $\tau \to 0$. Simultaneously, their suitably rescaled counterparts typically degenerate instead of converging to a non-spherical open set in the limit.

To be more specific, in axial coordinates $(x,z)\in\mathbb R\times\mathbb R^{n-1}$, with
$r=|z|$ and $p=n-2$, the Laplacian acting on axially symmetric functions becomes
$$
L_p=\partial_{xx}+\partial_{rr}+\frac{p}{r}\partial_r .
$$
See Section~\ref{subsec:axisymmetric-prelim} for more information. 
After translating the meridian domain to height $R=\tau^{-1}$, the weight 
$$r^p=R^p(1+\tau y)^p$$
appears in the equation.  Thus one must solve a weighted planar
conductivity problem, not the ordinary harmonic quadrature problem.  Moroover, the first
variation of the axial weight produces a $\sin t$-mode obstruction, with a leading term
$-\frac p4\,\tau\sin t.$
This mode cannot be removed by the high-mode outer correction, and we therefore introduce a second singular parameter $\delta$, which gives a controlled top--bottom imbalance of the singular measure.  Its leading contribution is
$$
4\delta\sin t+o(\delta),
$$
and $\delta$ is chosen to cancel the above-mentioned axial forcing.  The technical heart of the
proof is to show that solving the high modes does not feed a leading-order
error back into this scalar sine equation.  This is achieved by the weighted
elliptic trace estimates and singular perturbation bounds proved in
Section~\ref{sec:higher-dimensional-case}.

\begin{rem}
The reason we do not use the Wolff \cite{W1995} or Lewis--Vogel constructions \cite{LV1991,LV2001,LV2002} is that the present problem is sensitive to the exact boundary measure.  In an iterative Cantor-type construction, a weak limit of domains may preserve a distributional identity only with an abstract
limiting measure.  In other words, one may obtain
$$
\Delta u   =  c\,\mu-f(u)\mathbf 1_{\Omega}\,dx \qquad\text{in }\mathcal D'(\mathbb R^n),
$$
where $\mu$ is a singular limiting measure, instead of the required identity with
$$
    \mu=\mathcal H^{n-1}\!\lfloor_{\partial^*\Omega}.
$$
To recover \eqref{eq:main-higher-dimensional-law} by such a method, one would
need strict convergence of the perimeter measures, for example
$$
    \mathcal H^{n-1}(\partial\Omega_k)
    \to
    \mathcal H^{n-1}(\partial\Omega_\infty),
$$
together with corresponding convergence of the restricted boundary measures.
This appears highly nontrivial for the Cantor-type geometries produced by the Wolff--Lewis--Vogel mechanism.  Our axially symmetric construction circumvents this difficulty by generating the required boundary measure directly via an appropriately weighted quadrature identity. This approach, however, has the drawback that one can no longer ensure that $\Omega$ is homeomorphic to a ball, and the construction does not yield a sequence of compact families any more.
\end{rem}

\subsection{Organization of the paper}

Section~\ref{sec:preliminaries} collects the conformal mapping and harmonic analysis tools used throughout the paper, including Smirnov domains, inner--outer factorization, the Poisson balayage operator, and a small-data version of the Duren--Shapiro--Shields construction.  It also proves the semilinear disk Dirichlet estimates and the elementary rotational formulas used in the higher-dimensional constructions.  In Section~\ref{sec:smirnov-case} we prove Theorem~\ref{thm:smirnov-case}.  In Section~\ref{sec:non-smirnov-case} we prove the planar semilinear counterexample, Theorem~\ref{thm:counterexample-branch}.
In Section~\ref{sec:higher-dimensional-case} we prove the all-dimensional counterexample, Theorem~\ref{thm:higher-dimensional-counterexample}, by a weighted planar construction and axial rotation.

\medskip

\noindent{\bf Acknowledgement:} The author would like to express his sincere gratitude to Xavier Ros-Oton, who inspired him to investigate the present problem through discussions on harmonic quadrature identities. The author also thanks Dmitry Khavinson for drawing his attention to the literature on harmonic quadrature identities in planar domains, possibly with finite connectivity.

\section{Preliminaries}\label{sec:preliminaries}

Let us  consider the case when $\Omega\subset \mathbb R^2$ is a Jordan domain
with rectifiable boundary and satisfies \eqref{mean value Serrin}. Let
$\Gamma:=\partial\Omega$, and let
$$
\varphi:\mathbb D\to\Omega
$$
be a conformal map. By the Carath\'eodory--Osgood theorem, $\varphi$ extends
homeomorphically to $\overline{\mathbb D}$. Since $\Gamma$ is rectifiable, one has
$\varphi'\in H^1(\mathbb D)$; see \cite[Chapter~10, \S10.2]{D1970}. Equivalently,
the boundary parametrization
$$
z(t):=\varphi(e^{it}),\qquad 0\le t\le 2\pi,
$$
is absolutely continuous, and
$$
\int_0^{2\pi} |(\varphi')^*(e^{it})|\,dt<\infty,
\qquad
ds = |z_t(t)|\,dt = |(\varphi')^*(e^{it})|\,dt
\quad\text{for a.e. }t\in[0,2\pi],
$$
where $(\varphi')^*$ denotes the a.e. nontangential boundary trace of $\varphi'$.

Since $\varphi'$ has no zeros in $\mathbb D$, its canonical factorization has the form
$$
\varphi' = S\,\Phi,
$$
where $S$ is singular inner and $\Phi$ is outer. In accordance with \cite[Chapter~10, \S10.3]{D1970}, a \emph{Smirnov domain} can be characterized equivalently by the condition $S \equiv 1$ in the canonical factorization of $\varphi'$, or, equivalently, by the requirement that $\varphi'$ is an outer function.  
Rectifiable Jordan domains need not be Smirnov, and we refer to  e.g. \cite[Chapter~10]{D1970} for more information on Smirnov domains.

We first record the elementary implication from harmonic moments to the
measure-theoretic Bernoulli law.  This lemma is used for the quasidisks
constructed below. In this class the boundary values of continuous functions have
harmonic extensions with the Sobolev trace property used in the proof; see e.g. \cite{GH2012} for more information.

\begin{lem}\label{lem:quadrature-to-weak-serrin} 
Let $\Omega\subset\mathbb C$ be a bounded quasidisk  with rectifiable
boundary.  Let $g\in L^\infty(\Omega)$, and let $u\in W^{1,2}_0(\Omega)$ solve
\begin{equation*}
    -\Delta u=g \qquad\text{weakly in }\Omega.
\end{equation*}
Assume that there is a constant $c>0$ such that
\begin{equation}\label{eq:weighted-quadrature-general}
    \int_\Omega g h\,dA=c\int_{\partial\Omega}h\,d\mathcal H^1
\end{equation}
for every harmonic function $h\in C(\overline\Omega)$.  Suppose further that
for every $\varphi\in C_c^\infty(\mathbb C)$ the harmonic extension $H_\varphi$
of $\varphi|_{\partial\Omega}$ belongs to $W^{1,2}(\Omega)$ and
$\varphi-H_\varphi\in W^{1,2}_0(\Omega)$.  Then the zero extension of $u$ satisfies
\begin{equation}\label{eq:weak-law-from-quadrature}
    \Delta u=c\,\mathcal H^1\!\lfloor_{\partial\Omega}-g\mathbf 1_\Omega\,dA
    \qquad\text{in }\mathcal D'(\mathbb C).
\end{equation}
In particular, if $g=f(u)$, then $u$ satisfies the distributional Bernoulli law
with source $f(u)$.
\end{lem}

\begin{proof}
Let $\varphi\in C_c^\infty(\mathbb C)$, and let $H_\varphi$ be the harmonic
extension of $\varphi|_{\partial\Omega}$.  Set
$$
    v:=\varphi-H_\varphi.
$$
By assumption, $v\in W^{1,2}_0(\Omega)$, and since $H_\varphi$ is harmonic,
$\Delta v=\Delta\varphi$ in $\Omega$.  The weak equation $-\Delta u=g$ gives
$$
    \int_\Omega \nabla u\cdot\nabla v\,dA=\int_\Omega gv\,dA.
$$
Therefore, using $v\in W^{1,2}_0(\Omega)$,
\begin{align*}
    \int_{\mathbb C}u\Delta\varphi\,dA
    &=\int_\Omega u\Delta v\,dA
      =-\int_\Omega \nabla u\cdot\nabla v\,dA  \\
    &=-\int_\Omega g(\varphi-H_\varphi)\,dA
      =-\int_\Omega g\varphi\,dA+
        \int_\Omega gH_\varphi\,dA.
\end{align*}
Applying \eqref{eq:weighted-quadrature-general} to the harmonic function
$H_\varphi$ gives
$$
    \int_\Omega gH_\varphi\,dA
    =c\int_{\partial\Omega}H_\varphi\,d\mathcal H^1
    =c\int_{\partial\Omega}\varphi\,d\mathcal H^1.
$$
This is exactly \eqref{eq:weak-law-from-quadrature}.
\end{proof}

For a complex-valued function $f\in W^{1,1}_{\rm loc}(\Omega)$, we write
$$
\partial_z f = \frac12(f_x-if_y),
\qquad
\partial_{\bar z} f = \frac12(f_x+if_y).
$$
If $\theta\in\mathbb T$, then the directional derivative of $f$ in the direction $\theta$
is
\begin{equation}\label{directional derivative}
Df(z)\theta=\theta\,\partial_z f + \overline{\theta}\,\partial_{\bar z}f,
\end{equation}
which, when $f$ is real-valued, coincides with 
$2\operatorname{Re}\bigl(\theta\,\partial_z f\bigr).$
See \cite[Section~2.4 and Section~2.9]{AIM2009}.

We also recall Green's formula in complex notation
\cite[Theorem~2.9.1]{AIM2009}: if $U\subset\mathbb C$ is a bounded domain with
rectifiable boundary and $f,g\in W^{1,1}(U)\cap C(\overline U)$, then
\begin{equation}\label{green}
\int_U \bigl(\partial_z f+\partial_{\bar z}g\bigr)\,dA(z)
=
\frac{i}{2}\int_{\partial U} \bigl(f\,d\overline z-g\,dz\bigr).
\end{equation}

\subsection{Disk-side analytic tools}

We now collect the disk-side tools that will be used in the non-Smirnov construction.
Write the standard normalization
$$
dm(e^{it}) := \frac{dt}{2\pi}
\qquad\text{on }\mathbb T,
\qquad
 da(z):=\frac{dA(z)}{\pi}
\qquad\text{on }\mathbb D,
$$
so that $dm(\mathbb T)=da(\mathbb D)=1$. For $z\in\mathbb D$ and $\zeta\in\mathbb T$ we denote by
$$
P_z(\zeta):=\frac{1-|z|^2}{|\zeta-z|^2}
$$
the normalized Poisson kernel, so that $\int_{\mathbb T} P_z(\zeta)\,dm(\zeta)=1$. For $h\in L^1(\mathbb T,dm)$ we write
$$
P[h](z):=\int_{\mathbb T} P_z(\zeta)h(\zeta)\,dm(\zeta),
\qquad z\in\mathbb D,
$$
for its Poisson extension. For $G\in L^1(\mathbb D,da)$ we define the boundary balayage operator
\begin{equation}\label{eq:T-def}
(\ct G)(\zeta):=\int_{\mathbb D} P_z(\zeta)G(z)\,da(z),
\qquad \zeta\in\mathbb T.
\end{equation}
We emphasize that $\ct G$ is defined on $\mathbb T$, the boundary of $\mathbb D$.
Moreover, by Fubini's theorem,
\begin{equation}\label{eq:T-fubini}
\int_{\mathbb D} P[h](z)G(z)\,da(z)
=
\int_{\mathbb T} h(\zeta)(\ct G)(\zeta)\,dm(\zeta)
\end{equation}
whenever $h\in L^1(\mathbb T,dm)$ and $G\in L^1(\mathbb D,da)$.

For $0<\alpha<1$ we write $C^\alpha(\mathbb T)$ for the real H\"older space on $\mathbb T$, endowed with the norm
$$
\|h\|_{C^\alpha(\mathbb T)}
:=
\|h\|_{L^\infty(\mathbb T)}
+
\sup_{\xi\neq\eta}\frac{|h(\xi)-h(\eta)|}{|\xi-\eta|^\alpha}.
$$
We also use the $4$-fold symmetric mean-zero real-valued function subspace
\begin{equation}\label{eq:Xalpha4}
X_\alpha^4
:=
\Bigl\{h\in C^\alpha(\mathbb T;\,\mathbb R): h(i\zeta)=h(\zeta)\text{ for all }\zeta\in\mathbb T,
\ \int_{\mathbb T} h\,dm=0\Bigr\}.
\end{equation}

\begin{lem}\label{lem:T-holder}
Let $0<\alpha<1$. Then the linear operator $\ct$ defined in \eqref{eq:T-def} maps $L^\infty(\mathbb D)$ boundedly into $C^\alpha(\mathbb T)$. More precisely, there exists $C_\alpha>0$ such that
\begin{equation}\label{eq:T-holder-bound}
\|\ct G\|_{C^\alpha(\mathbb T)}
\le C_\alpha \|G\|_{L^\infty(\mathbb D)}
\qquad\text{for every }G\in L^\infty(\mathbb D).
\end{equation}
Moreover, for every $0<\delta<1$,
\begin{equation}\label{eq:T-strip-bound}
\bigl\|\ct(\mathbf 1_{\{1-\delta<|z|<1\}}G)\bigr\|_{C^\alpha(\mathbb T)}
\le C_\alpha \|G\|_{L^\infty(\mathbb D)}\,\delta^{1-\alpha}\log\frac{2}{\delta}.
\end{equation}
\end{lem}

\begin{proof}
Write $z=re^{i\theta}$ and $\zeta=e^{i\phi}$, and in polar coordinates,
$$
P_r(t):=\frac{1-r^2}{1-2r\cos t+r^2}.
$$
Then
$$
(\ct G)(e^{i\phi})
=
\int_0^1\int_0^{2\pi} P_r(\phi-\theta)G(re^{i\theta})\,dm(e^{i\theta})\,2rdr.
$$
The $L^\infty$ bound in \eqref{eq:T-holder-bound} is immediate from positivity of the Poisson kernel and the normalization $\int_{\mathbb T}P_r(t)\,dm(e^{it})=1$.

To estimate the H\"older-seminorm, let us fix $h\in\mathbb R$. Since
$$
\|P_r(\cdot+h)-P_r\|_{L^1(\mathbb T,dm)}\le 2,
$$
and, by the fundamental theorem of calculus,
$$
\|P_r(\cdot+h)-P_r\|_{L^1(\mathbb T,dm)}\le |h|\int_0^1 \|P_{r}'(\cdot+sh)\|_{L^1(\mathbb T,\,dm)}\,ds 
\le |h|\,\|P_r'\|_{L^1(\mathbb T,dm)},
$$
it is enough to bound $\|P_r'\|_{L^1}$. A direct calculation gives
$$
P_r'(t)=-\frac{2r(1-r^2)\sin t}{(1-2r\cos t+r^2)^2},
$$
so the standard estimate
\begin{equation}\label{eq:poisson-derivative-bound}
\|P_r'\|_{L^1(\mathbb T,dm)}\le \frac{C}{1-r}
\end{equation}
holds for some absolute constant $C>0$. Therefore
\begin{equation}\label{eq:poisson-difference-bound}
\|P_r(\cdot+h)-P_r\|_{L^1(\mathbb T,dm)}
\le C\min\!\left\{1,\frac{|h|}{1-r}\right\}.
\end{equation}
Using \eqref{eq:poisson-difference-bound} and the definition of $\ct G$,
\begin{align*}
|\ct G(e^{i(\phi+h)})-\ct G(e^{i\phi})|
&\le \|G\|_{L^\infty(\mathbb D)}
\int_0^1 2r\,\|P_r(\cdot+h)-P_r\|_{L^1(\mathbb T,dm)}\,dr \\
&\le C\|G\|_{L^\infty(\mathbb D)}
\int_0^1 2r\min\!\left\{1,\frac{|h|}{1-r}\right\}dr \\
&\le C\|G\|_{L^\infty(\mathbb D)}\,|h|\log\frac{2}{|h|}.
\end{align*}
Since $|h|\log(2/|h|)\le C_\alpha |h|^\alpha$ for $0<\alpha<1$, this proves \eqref{eq:T-holder-bound}.

To prove \eqref{eq:T-strip-bound}, set $A_\delta:=\{z\in\mathbb D:1-\delta<|z|<1\}$. The same argument for the $L^{\infty}$-bound in \eqref{eq:T-holder-bound}  yields
$$
\|\ct(\mathbf 1_{A_\delta}G)\|_{L^\infty(\mathbb T)}
\le \|G\|_{L^\infty(\mathbb D)}\,da(A_\delta)
\le 2\delta\,\|G\|_{L^\infty(\mathbb D)}.
$$
Now if $|h|\ge \delta$, then the  estimate above gives directly
$$
|\ct(\mathbf 1_{A_\delta}G)(e^{i(\phi+h)})-\ct(\mathbf 1_{A_\delta}G)(e^{i\phi})|
\le 4\delta\,\|G\|_{L^\infty(\mathbb D)}.
$$
If $|h|<\delta$, then by \eqref{eq:poisson-difference-bound}
\begin{align*}
 & \quad   |\ct(\mathbf 1_{A_\delta}G)(e^{i(\phi+h)})-\ct(\mathbf 1_{A_\delta}G)(e^{i\phi})|  
 \le C\|G\|_{L^\infty(\mathbb D)}
\int_{1-\delta}^1 2r\min\!\left\{1,\frac{|h|}{1-r}\right\}dr \\
&  \le C\|G\|_{L^\infty(\mathbb D)} \int_{0}^{\delta} \min\!\left\{1,\frac{|h|}{r}\right\}dr  \le C\|G\|_{L^\infty(\mathbb D)}\,|h|\left(1+\log\frac{\delta}{|h|}\right)\\
&\le C  \|G\|_{L^\infty(\mathbb D)} |h|^{\alpha} |h|^{1-\alpha}\left(1+\log\frac{\delta}{|h|}\right)\le C_\alpha \|G\|_{L^\infty(\mathbb D)} |h|^{\alpha}  \delta^{1-\az} .
\end{align*}
Thus in both cases,
$$
\frac{|\ct(\mathbf 1_{A_\delta}G)(e^{i(\phi+h)})-\ct(\mathbf 1_{A_\delta}G)(e^{i\phi})|}{|h|^\alpha}
\le C_\alpha\|G\|_{L^\infty(\mathbb D)}\,\delta^{1-\alpha}\log\frac{2}{\delta},
$$
which proves \eqref{eq:T-strip-bound}.
\end{proof}

\begin{lem}\label{lem:T-continuity}
Let $0<\alpha<1$, and let $\{G_j\}_{j\in\nn}$ be a sequence in $L^\infty(\mathbb D)$ such that
$$
\sup_j\|G_j\|_{L^\infty(\mathbb D)}<\infty.
$$
Assume that $G_j\to G$ locally uniformly in $\mathbb D$. Then
$$
\ct G_j\to \ct G
\qquad\text{in }C^\alpha(\mathbb T).
$$
The same conclusion remains true if $G_j$ is replaced by $G_jQ_j$, provided $\{Q_j\}$ is uniformly bounded in $L^\infty(\mathbb D)$ and converges locally uniformly.
\end{lem}

\begin{proof}
Fix $0<\delta<1$. On the compact disk $(1-\delta) \overline{\mathbb D}$, the convergence is uniform, so by \eqref{eq:T-holder-bound}
$$
\bigl\|\ct(\mathbf 1_{\{|z|\le 1-\delta\}}(G_j-G))\bigr\|_{C^\alpha(\mathbb T)}\to 0.
$$
For the boundary strip $A_\delta=\{1-\delta<|z|<1\}$, \eqref{eq:T-strip-bound} gives
$$
\sup_j\bigl\|\ct(\mathbf 1_{A_\delta}(G_j-G))\bigr\|_{C^\alpha(\mathbb T)}
\le C_\alpha\sup_j\|G_j-G\|_{L^\infty(\mathbb D)}\,\delta^{1-\alpha}\log\frac{2}{\delta}
\le C_\alpha'\delta^{1-\alpha}\log\frac{2}{\delta}.
$$
Since $\delta$ is arbitrary, the claim follows. The final statement can be proved in exactly the same way, since the product of two locally uniformly convergent bounded sequences converges locally uniformly and remains uniformly bounded.
\end{proof}

\subsection{The Duren--Shapiro--Shields construction}

We next recall a version of the Duren--Shapiro--Shields construction \cite{DSS1966}. Towards this, we firstly introduce the Herglotz transform associated with a finite real Borel measure $\nu$ on $\mathbb T$:
\begin{equation}\label{eq:Herglotz-transform}
\mathcal H[\nu](z)
:=
\int_{\mathbb T}\frac{\zeta+z}{\zeta-z}\,d\nu(\zeta),
\qquad z\in\mathbb D.
\end{equation}
Its real part is the Poisson integral of $\nu$,
\begin{equation}\label{eq:Herglotz-real-part}
\operatorname{Re}\mathcal H[\nu](z)=P[\nu](z)
:=
\int_{\mathbb T} P_z(\zeta)\,d\nu(\zeta).
\end{equation}
If $d\nu=w\,dm+d\mu$ with $w\in L^1(\mathbb T,dm)$ real-valued and $\mu$ singular, then the inner--outer factorization of $\exp(-\mathcal H[\nu])$ shows that its radial boundary values satisfy
\begin{equation}\label{eq:boundary-modulus-Hnu}
\bigl|\exp(-\mathcal H[\nu])^*(\zeta)\bigr|=e^{-w(\zeta)}
\qquad\text{for }m\text{-a.e. }\zeta\in\mathbb T,
\end{equation}
where $^*$ denotes the radial limit. 
See \cite[Chapter~2 and Chapter~10]{D1970} and \cite[Chapter~7]{P1992}.

We also recall the Zygmund class $A^*$. A continuous $2\pi$-periodic function $\psi$ belongs to $A^*$ if there exists $A>0$ such that
\begin{equation}\label{eq:Astar-function}
|\psi(t+h)+\psi(t-h)-2\psi(t)|\le A|h|
\qquad\text{for all }t\in\mathbb R,\ h\in\mathbb R.
\end{equation}
In particular, 
\begin{equation}\label{eq:Astar-continuity}
|\psi(t+h)-\psi(t)|\le Ch\log\frac 2 h;
\end{equation}
see \cite[Page 248]{DSS1966} for more information.

Following the definition in \cite{DSS1966}, a finite \emph{real Borel measure} $\nu$ on $\mathbb T$ is said to be of \emph{class $A^*$} if one (equivalently every) cumulative distribution function of $\nu$ satisfies \eqref{eq:Astar-function}. In particular, if $w\in L^\infty(\mathbb T)$ and
$$
V(t):=\int_0^t w(e^{is})\,ds,
$$
then
$$
|V(t+h)-2V(t)+V(t-h)|
\le 2\|w\|_{L^\infty(\mathbb T)}|h|,
$$
so $w\,dm$ is automatically of class $A^*$. Moreover,  there exists  a nonzero positive singular measure $\sigma$ of class $A^*$ on $\mathbb T$; see particularly the discussion following \cite[Theorem~1]{DSS1966}. These facts will be used in the construction of the non-Smirnov domain, and we refer to \cite[Section 7.3]{P1992} for more information. 

The following lemma presents the small-data version of the construction of the class of domains introduced by Duren, Shapiro, and Shields (abbreviated as DSS), and we say these domains are in DSS-class. For the reader's convenience, we restate the argument of DSS within our present framework. The main idea of their construction is to use  the Ahlfors--Weill theorem (see \cite[Chapter~II, \S5.1]{L1987} or \cite[Theorem~1.12]{P1992}) with small boundary data. 

\begin{lem}[\cite{DSS1966}]\label{lem:small-data-DSS}
There exists $\epsilon_0>0$ with the following property. Let 
$$
d\nu : = w\,dm + d\mu
$$
be a finite real Borel measure on $\mathbb T$, where $w\in L^\infty(\mathbb T)$ is real-valued and $\mu$ is a nonnegative singular measure of class $A^*$. Assume that
$$
\|w\|_{L^\infty(\mathbb T)} + \|\mu\|_{A^*} \le \epsilon_0.
$$
Define
\begin{equation}\label{eq:f-nu-prime}
F_\nu(z):=\mathcal H[\nu](z),
\qquad
f_\nu(z):=\int_0^z e^{-F_\nu(\xi)}\,d\xi.
\end{equation}
Then $f_\nu$ is conformal in $\mathbb D$, extends homeomorphically to $\overline{\mathbb D}$, and maps $\mathbb D$ onto a rectifiable Jordan domain. Moreover, if $\mu\neq 0$, then the image domain is not Smirnov.
\end{lem}

\begin{proof}
\cite[Formula (6)-(11), Proof of Theorem~1]{DSS1966} shows that the $A^*$ condition for a measure $\nu$ is equivalent to an estimate of the form
\begin{equation}\label{eq:DSS-Fprime-bound}
|F_\nu'(z)|\le \frac{C\|\nu\|_{A^*}}{1-|z|}, 
\qquad z\in\mathbb D,
\end{equation}
where $C>0$ is an absolute constant and $\|\nu\|_{A^*}$ denotes any fixed equivalent $A^*$ norm of the cumulative function of $\nu$. Then by Cauchy's estimate applied to $F_\nu'$ on the circle of radius $\frac{1+|z|}{2}$,
\begin{equation}\label{eq:DSS-Fsecond-bound}
|F_\nu''(z)|\le \frac{C\|\nu\|_{A^*}}{(1-|z|)^2}, 
\qquad z\in\mathbb D.
\end{equation}
Since $f_\nu'(z)=e^{-F_\nu(z)}$, its Schwarzian derivative is
$$
S_{f_\nu}=F_\nu''-\frac12(F_\nu')^2.
$$
Combining \eqref{eq:DSS-Fprime-bound} and \eqref{eq:DSS-Fsecond-bound}, we obtain
$$
(1-|z|^2)^2 |S_{f_\nu}(z)|
\le C\|\nu\|_{A^*} + C\|\nu\|_{A^*}^2.
$$
If $\epsilon_0>0$ is chosen small enough, then one gets
$$
\sup_{z\in\mathbb D}(1-|z|^2)^2 |S_{f_\nu}(z)|< 2.
$$

Now one can apply the Ahlfors--Weill--Nehari criterion to conclude that $f_\nu$ is univalent in $\mathbb D$ and admits a quasiconformal extension to the sphere. In particular, $f_\nu$ extends homeomorphically to $\overline{\mathbb D}$, and $f_\nu(\mathbb D)$ is a Jordan domain.

The rectifiability criterion in \cite[Page 250, Proof of Theorem 1]{DSS1966} now applies: Since the singular part $\mu$ is nonnegative and $e^{-w}\in L^1(\mathbb T,dm)$ (indeed $w\in L^\infty$), the boundary curve of $f_\nu(\mathbb D)$ is rectifiable. Finally, \cite[p.~247]{DSS1966} states that the presence of a non-trivial singular factor is a property of the image domain itself. Therefore, the image domain is not Smirnov as long as  $\mu\neq 0$.
\end{proof}

\subsection{A semilinear Dirichlet estimate on the disk}

We shall also need a small-amplitude semilinear estimate.  The point is that,
after scaling the physical domain by a factor $\lambda>0$, the unknown has size
$O(\lambda^2)$, so the source $f(u)$ can be regarded a perturbation of the positive constant
$f(0)$.

\begin{lem}\label{lem:semilinear-disk-estimate}
Let $f\in C^2((-\rho,\rho))$ and set $f_0:=f(0)>0$.  Fix
$0<\alpha<1$.  Let $\mu$ be a four-fold symmetric singular measure, and define, for $W\in X_\alpha^4$ and
$a\ge0$,
\begin{equation}\label{eq:G-W-a-prelim}
    G_{W,a}(z):=
    \exp\{-2(P[W](z)+aP[\mu](z))\},
    \qquad z\in\mathbb D.
\end{equation}
There are constants $r_0>0$, $a_0>0$, and $\lambda_0>0$ such that, whenever
$\|W\|_{X_\alpha^4}\le r_0$, $0\le a\le a_0$, and
$0\le\lambda\le\lambda_0$, the problem
\begin{equation}\label{eq:semilinear-disk-dirichlet}
    -\Delta v=G_{W,a}(z)f(\lambda^2v)
    \quad\text{in }\mathbb D,
    \qquad
    v=0\quad\text{on }\partial\mathbb D
\end{equation}
has a unique solution $v_{W,a,\lambda}\in L^\infty(\mathbb D)\cap
W^{1,2}_0(\mathbb D)$.  Moreover, uniformly for these parameters,
\begin{equation}\label{eq:semilinear-source-small}
    \|f(\lambda^2v_{W,a,\lambda})-f_0\|_{L^\infty(\mathbb D)}
    \le C\lambda^2.
\end{equation}
The map $W\mapsto v_{W,a,\lambda}$ is Fr\'echet differentiable, and for every
$H\in X_\alpha^4$,
\begin{equation}\label{eq:semilinear-source-derivative-small}
    \bigl\|D_W[f(\lambda^2v_{W,a,\lambda})][H]\bigr\|_{L^\infty(\mathbb D)}
    \le C\lambda^2\|H\|_{X_\alpha^4}.
\end{equation}
Consequently,
\begin{equation}\label{eq:T-semilinear-perturbation}
    \bigl\|\ct(G_{W,a}f(\lambda^2v_{W,a,\lambda}))
    - f_0\ct(G_{W,a})\bigr\|_{C^\alpha(\mathbb T)}
    \le C\lambda^2,
\end{equation}
and the same $O(\lambda^2)$ estimate holds for the Fr\'echet derivative with respect to $W$.
Here $C$ depends only on $\|f\|_{C^2}$ and $r_0.$
\end{lem}

\begin{proof}
Let $\mathfrak G$ denote the Green operator for the Dirichlet Laplacian on
$\mathbb D$, so that $\mathfrak G F$ is the unique solution of
$-\Delta w=F$ in $\mathbb D$ and $w=0$ on $\partial\mathbb D$.  The Green
kernel is positive and integrable; hence
\begin{equation}\label{eq:green-linfty-bound}
    \|\mathfrak G F\|_{L^\infty(\mathbb D)}
    \le C_{\mathbb D}\|F\|_{L^\infty(\mathbb D)}.
\end{equation}
For $\|W\|_{X_\alpha^4}\le r_0$ and $a\ge0$, the density $G_{W,a}$ defined via \eqref{eq:G-W-a-prelim} satisfies
$$
    0<G_{W,a}(z)\le e^{2\|W\|_{L^\infty(\mathbb T)}}\le C_0 \qquad\text{in }\mathbb D,
$$
for some $C_0=C_0(r_0)$ as $P[\mu]\ge0$.  Choose $R>0$ so large that
$$C_{\mathbb D}C_0\sup_{|s|<\rho}|f(s)|\le R$$
after restricting to
$|s|\le \rho/2$, and then choose $\lambda_0>0$ so small that
$\lambda_0^2R<\rho/2$ and
$$
    C_{\mathbb D}C_0\lambda_0^2\sup_{|s|<\rho}|f'(s)|<1.
$$
For $v\in L^\infty(\mathbb D)$ with $\|v\|_\infty\le R$, define
$$
    \mathcal A(v):=\mathfrak G\bigl(G_{W,a}f(\lambda^2v)\bigr).
$$
The  choice on $\lambda_0$ imply that $\mathcal A$ maps the closed ball
$\{\|v\|_\infty\le R\}$ into itself and is a contraction.  Hence
\eqref{eq:semilinear-disk-dirichlet} has a unique bounded solution.  Since the
right-hand side is bounded, the solution  belongs to $W^{1,2}_0(\mathbb D)$.

The bound \eqref{eq:semilinear-source-small} follows from
$\|v_{W,a,\lambda}\|_\infty\le R$ and the Lipschitz regularity of $f$
$$
    |f(\lambda^2v_{W,a,\lambda})-f(0)|\le C\lambda^2R\le C\lambda^2.
$$
To prove the Fr\'echet differentiability, let $H\in X_\alpha^4$ and let
$\dot v=D_W(v_{W,a,\lambda})[H]$.  Differentiating
\eqref{eq:semilinear-disk-dirichlet} gives
\begin{equation}\label{eq:linearized-v}
    -\Delta \dot v
    =-2G_{W,a}P[H]f(\lambda^2v_{W,a,\lambda})
     +\lambda^2G_{W,a}f'(\lambda^2v_{W,a,\lambda})\dot v,
    \qquad \dot v|_{\partial\mathbb D}=0.
\end{equation}
Using \eqref{eq:green-linfty-bound} and the smallness of $\lambda$, we absorb the
last term on the right and obtain
$$
    \|\dot v\|_{L^\infty(\mathbb D)}
    \le C\|P[H]\|_{L^\infty(\mathbb D)}
    \le C\|H\|_{X_\alpha^4}.
$$
Thus
$$
    \bigl|D_W[f(\lambda^2v_{W,a,\lambda})][H]\bigr|
    =\lambda^2 |f'(\lambda^2v_{W,a,\lambda})\dot v|
    \le C\lambda^2\|H\|_{X_\alpha^4},
$$
which proves \eqref{eq:semilinear-source-derivative-small}.  Moreover, the implicit differentiation argument for contractions justifies the Fr\'echet
differentiability used above; see, for instance, the contraction mapping principle
in \cite[Chapter~5]{GT2001}.

Finally, \eqref{eq:T-semilinear-perturbation} follows from Lemma~\ref{lem:T-holder}:
$$
\begin{aligned}
&\bigl\|\ct(G_{W,a}f(\lambda^2v_{W,a,\lambda}))
 -f_0\ct(G_{W,a})\bigr\|_{C^\alpha(\mathbb T)} \\
&\qquad\le C_\alpha
\|G_{W,a}(f(\lambda^2v_{W,a,\lambda})-f_0)\|_{L^\infty(\mathbb D)}
\le C\lambda^2.
\end{aligned}
$$
The Fr\'echet derivative estimate follows by differentiating the product
$G_{W,a}f(\lambda^2v_{W,a,\lambda})$ and using
\eqref{eq:semilinear-source-derivative-small} together with
$D_W(G_{W,a})[H]=-2G_{W,a}P[H]$.
\end{proof}

\subsection{Axisymmetric reduction and rotation}\label{subsec:axisymmetric-prelim}

We shall use the following elementary reduction in dimensions $n\ge3$.  Write a
point of $\mathbb R^n$ as
$$
    X=(x,z),\qquad x\in\mathbb R,\qquad z\in\mathbb R^{n-1},
$$
and set $r=|z|$.  Let
$ p:=n-2.$ If $V=V(x,r)$ is an axially symmetric function, then
\begin{equation}\label{eq:Lp-def}
    \Delta_{\mathbb R^n}V
    =V_{xx}+V_{rr}+\frac{p}{r}V_r
    =:L_pV.
\end{equation}
Equivalently,
$$
    L_pV=r^{-p}\operatorname{div}_{x,r}(r^p\nabla_{x,r}V).
$$

Let $D\subset\subset\{(x,r):r>0\}$ be a bounded planar domain with rectifiable
boundary.  Its rotation is
$$
    \Omega_D:=\{(x,z)\in\mathbb R\times\mathbb R^{n-1}:(x,|z|)\in D\}.
$$
For an integrable function $H$ on $\Omega_D$, define its spherical average by
$$
    \widetilde H(x,r):=\bint_{\mathbb S^{n-2}}H(x,r\omega)\,d\omega.
$$
Then polar coordinates in the $z$-variables give
\begin{equation}\label{eq:rotation-volume-formula}
    \int_{\Omega_D}H\,dX
    =\mathcal H^{n-2}(\mathbb S^{n-2})\int_D \widetilde H(x,r)r^p\,dx\,dr.
\end{equation}
When $\partial D$ is rectifiable, then the area formula applied to the parametrization
$(s,\omega)\mapsto (x(s),r(s)\omega)$ yields
\begin{equation}\label{eq:rotation-boundary-formula}
    \int_{\partial^*\Omega_D}H\,d\mathcal H^{n-1}
    =\mathcal H^{n-2}(\mathbb S^{n-2})\int_{\partial D}\widetilde H(x,r)r^p\,ds.
\end{equation}
Here the use of the reduced boundary is the natural finite-perimeter
interpretation of the rotated rectifiable surface; see, for instance, the area
formula and the structure theorem for sets of finite perimeter in
\cite{AFP2000, M2012}.

If $H$ is harmonic near $\overline{\Omega}_D$, then its spherical average
$\widetilde H$ satisfies $L_p\widetilde H=0$ near $\overline D$.  Consequently, if
a bounded function $q$ on $D$ and a constant $c>0$ satisfy
\begin{equation}\label{eq:meridian-weighted-moment}
\int_D qv\,r^p\,dx\,dr
=c\int_{\partial D}v\,r^p\,ds
\end{equation}
for every admissible variational $L_p$-harmonic replacement $v$ in $D$, then the
rotated density $Q(x,z):=q(x,|z|)$ satisfies the corresponding harmonic moment identity after rotation.

Recall that quasidisks are Sobolev extension domains; see e.g. \cite{J1981} as well as some recent progress \cite{KRZ2025, KRZ2026}. Therefore domains constructed by Lemma~\ref{lem:small-data-DSS} satisfies the conditions in the following lemma. 

\begin{lem}\label{lem:rotated-moment-to-weak-law}
Let $D\subset\subset\{r>0\}$ be a bounded finite-perimeter meridian domain for which (variational) $L_p$-harmonic  functions with the relevant Sobolev boundary traces are well-defined.  Let $q\in L^\infty(D)$ be nonnegative, and assume that
\eqref{eq:meridian-weighted-moment} holds for every such variational
$L_p$-harmonic replacement $v$.  Let $u\in W^{1,2}_0(D;r^pdxdr)$ solve
$$
    -L_pu=q\quad\text{in }D.
$$
Set
$$
\Omega_D:=\{(x,z)\in\mathbb R\times\mathbb R^{n-1}:(x,|z|)\in D\},
\qquad
U(x,z):=u(x,|z|),
$$
and $Q(x,z):=q(x,|z|)$.  Extend $U$ by zero outside $\Omega_D$.  Then
\begin{equation}\label{eq:rotated-weak-law}
\Delta U
=c\,\mathcal H^{n-1}\!\lfloor_{\partial^*\Omega_D}
-Q\mathbf 1_{\Omega_D}\,dX
\qquad\text{in }\mathcal D'(\mathbb R^n),
\end{equation}
with the same dimensional normalization as in
\eqref{eq:rotation-volume-formula} and \eqref{eq:rotation-boundary-formula}.
\end{lem}

\begin{proof}
Let $\varphi\in C_c^\infty(\mathbb R^n)$ and define its spherical average
$$
    \widetilde\varphi(x,r)
    :=\bint_{\mathbb S^{n-2}}\varphi(x,r\omega)\,d\omega .
$$
Since $U$ and $Q$ are axially symmetric,
$$
\int_{\Omega_D}U\Delta\varphi\,dX
=\mathcal H^{n-2}(\mathbb S^{n-2})
\int_D uL_p\widetilde\varphi\,r^p\,dx\,dr .
$$
Let $h$ be the variational $L_p$-harmonic replacement in $D$ whose boundary trace
equals the trace of $\widetilde\varphi$ on $\partial D$.  Then
$$
    \eta:=\widetilde\varphi-h\in W^{1,2}_0(D;r^pdxdr).
$$
Since $L_ph=0$ weakly,
$$
    L_p\widetilde\varphi=L_p\eta.
$$
Using the weak equation $-L_pu=q$ and the zero trace of $\eta$, we get
$$
\int_D uL_p\widetilde\varphi\,r^p
=\int_D uL_p\eta\,r^p
=-\int_D r^p\nabla u\cdot\nabla\eta
=-\int_D q\eta\,r^p.
$$
Hence
$$
\int_D uL_p\widetilde\varphi\,r^p
=-\int_D q\widetilde\varphi\,r^p+
  \int_D qh\,r^p.
$$
By \eqref{eq:meridian-weighted-moment},
$$
\int_D qh\,r^p
=c\int_{\partial D}h\,r^p\,ds
=c\int_{\partial D}\widetilde\varphi\,r^p\,ds,
$$
since $h$ and $\widetilde\varphi$ have the same trace on $\partial D$.  Therefore
$$
\int_D uL_p\widetilde\varphi\,r^p
=-\int_D q\widetilde\varphi\,r^p
+c\int_{\partial D}\widetilde\varphi\,r^p\,ds.
$$
Rotating back by \eqref{eq:rotation-volume-formula} and
\eqref{eq:rotation-boundary-formula} gives
$$
\int_{\mathbb R^n}U\Delta\varphi\,dX
=-\int_{\Omega_D}Q\varphi\,dX
+c\int_{\partial^*\Omega_D}\varphi\,d\mathcal H^{n-1},
$$
which is \eqref{eq:rotated-weak-law}.
\end{proof}

\section{Smirnov case in the plane}\label{sec:smirnov-case}

Since \eqref{mean value Serrin} is invariant under translation, we may assume that
$$
\int_\Gamma z\,ds = 0.
$$

If $f\in C(\overline\Omega)$ is holomorphic in $\Omega$, then its real and imaginary
parts are harmonic. Hence \eqref{mean value Serrin} implies
\begin{equation}\label{mean value holo}
c\int_\Gamma f\,ds = \int_\Omega f\,dA.
\end{equation}
In particular, for every analytic polynomial $\mathcal P$,
$$
c\int_\Gamma \mathcal P(z)\,ds = \int_\Omega \mathcal P(z)\,dA(z).
$$
Applying \eqref{green} to $f\equiv 0$ and $g(z)=\mathcal P(z)\overline z$, we also get
$$
\int_\Omega \mathcal P(z)\,dA(z)
=
\int_\Omega \partial_{\bar z}\bigl(\mathcal P(z)\overline z\bigr)\,dA(z)
=
\frac{1}{2i}\int_\Gamma \mathcal P(z)\,\overline z\,dz.
$$
Combining the two identities gives
\begin{equation}\label{orthogonal}
0
=
\int_\Gamma \mathcal P(z)\left(c\,ds-\frac{1}{2i}\,\overline z\,dz\right)
=
\int_\Gamma \mathcal P(z)\,g(z)\,ds,
\end{equation}
where
$$
g(z):=c-\frac{1}{2i}\,\overline z\,\tau(z),
\qquad
\tau(z):=\frac{dz}{ds}
$$
is the positively oriented unit tangent to $\Gamma$, defined for $ds$-a.e.
$z\in\Gamma$.

Since $dz=\tau\,ds$ and $|\tau|=1$, \eqref{orthogonal} is equivalently written as
\begin{equation}\label{eq:smirnov-cauchy-orthogonality}
\int_\Gamma \mathcal P(z)\,b(z)\,dz=0
\qquad\text{for every analytic polynomial } \mathcal P,
\end{equation}
where
$$
b(z):=\overline z-2ic\,\overline{\tau(z)}.
$$
This is exactly the boundary relation observed by Khavinson in the smooth/Smirnov
setting; see \cite[Theorem~1(ii)]{K1987}.

It is convenient to use Duren's notation
$$
E^1(\Omega)
:=
\Bigl\{F\in \mathcal O(\Omega): (F\circ\varphi)\varphi'\in H^1(\mathbb D)\Bigr\},
$$
which is independent of the particular Riemann map $\varphi$; see
\cite[Chapter~10]{D1970}. 

Now we are ready to prove Theorem~\ref{thm:smirnov-case}, which uses the non-tangential limit of Hardy functions in the unit disk $\mathbb D$; see also \cite{DZ2025}.

\begin{proof}[Proof of Theorem~\ref{thm:smirnov-case}]
Let $\Gamma:=\partial\Omega$, and let $\varphi:\mathbb D\to\Omega$ be the Riemann map
fixed above. By \eqref{eq:smirnov-cauchy-orthogonality} and
\cite[Theorem~10.4]{D1970}, the Cauchy integral of $b$ defines a holomorphic function
$F\in E^1(\Omega)$ whose nontangential boundary values satisfy
\begin{equation}\label{eq:smirnov-F-boundary}
F^*(z)=b(z)=\overline z-2ic\,\overline{\tau(z)}
\qquad\text{for }ds\text{-a.e. }z\in\Gamma.
\end{equation}
Equivalently,
\begin{equation}\label{eq:smirnov-F-E1}
(F\circ\varphi)\varphi' \in H^1(\mathbb D).
\end{equation}

Since $\Omega$ is Smirnov, $\varphi'$ is outer; see \cite[Chapter~10,
\S10.3]{D1970}. Hence \eqref{eq:smirnov-F-E1} implies that
$F\circ\varphi\in N^+$.
Here $N^+$ denotes the Smirnov class in $\mathbb D$; see \cite[Section~2.5]{D1970}.

On the other hand, \eqref{eq:smirnov-F-boundary} shows that
$$
(F\circ\varphi)^*(e^{it})
=
\overline{z(t)}-2ic\,\overline{\tau(t)}
\qquad\text{for a.e. }t\in[0,2\pi],
$$
so $(F\circ\varphi)^*\in L^\infty(\mathbb T)$. Therefore
$$
F\circ\varphi \in N^+\cap L^\infty(\mathbb T),
$$
and \cite[Theorem~2.11]{D1970} yields
\begin{equation}\label{eq:smirnov-F-bounded}
F\circ\varphi \in H^\infty(\mathbb D).
\end{equation}

Let $G$ be a primitive of $F$ in $\Omega$, so that $G'=F$. Then
$$
(G\circ\varphi)' = (F\circ\varphi)\varphi' \in H^1(\mathbb D).
$$
By \cite[Theorem~3.11]{D1970}, $G\circ\varphi$ extends continuously to
$\overline{\mathbb D}$ and its boundary values on $\mathbb T$ are absolutely continuous.

\noindent{\bf Step 1: A function $u$ satisfies Serrin's overdetermined system}. Define
$$
u(z):=\frac12\operatorname{Re}G(z)-\frac14|z|^2 + C_0,
$$
where $C_0\in\mathbb R$ will be chosen later. Since $G$ is holomorphic,
\begin{equation}\label{eq:smirnov-Delta-u}
\Delta u = -1 \qquad\text{in }\Omega.
\end{equation}
Also,
\begin{equation}\label{eq:smirnov-dz-u}
4\partial_z u = F-\overline z.
\end{equation}

We next show that the boundary trace of $u$ is constant. Set
$$
U(t):=u(z(t))=u(\varphi(e^{it})).
$$
Since $G\circ\varphi$ is absolutely continuous on $\mathbb T$ and $\partial\Omega$ is rectifiable, which ensures the absolute continuity of the mapping $t \mapsto \varphi(e^{it})$, it follows that the function $U$ is absolutely continuous on $[0,2\pi]$. For a.e. $t$,
$$
U'(t)= D u(z(t))\cdot z_t(t)=|z_t(t)|\,\partial_\tau u(z(t)).
$$
By \eqref{directional derivative},
$$
\partial_\tau u = 2\operatorname{Re}\bigl(\tau\,\partial_z u\bigr).
$$
Using \eqref{eq:smirnov-dz-u} and \eqref{eq:smirnov-F-boundary}, we obtain
$$
2\partial_z u = -ic\,\overline\tau
\qquad\text{for }ds\text{-a.e. on }\Gamma,
$$
hence
$$
\partial_\tau u
=2\operatorname{Re}\bigl(\tau\,\partial_z u\bigr)
=\operatorname{Re}\bigl(\tau(2\partial_z u)\bigr)
=\operatorname{Re}(-ic)=0
\qquad\text{for }ds\text{-a.e. on }\Gamma.
$$
Therefore $U'(t)=0$ for a.e. $t$. Since $U$ is absolutely continuous, it must be constant.
Thus by choosing $C_0$ suitably, we may assume that
\begin{equation}\label{eq:smirnov-u-zero-boundary}
u=0 \qquad\text{on }\Gamma.
\end{equation}

Let $\nu_{\rm in}=i\tau$ and $\nu_{\rm out}=-i\tau$ denote the inward and outward unit
normals to $\Gamma$. Again by \eqref{eq:smirnov-dz-u} and
\eqref{eq:smirnov-F-boundary},
$$
\partial_{\nu_{\rm in}}u
=2\operatorname{Re}\bigl(\nu_{\rm in}\partial_z u\bigr)
=\operatorname{Re}\bigl(i\tau(2\partial_z u)\bigr)
=\operatorname{Re}\bigl(i\tau(-ic\,\overline\tau)\bigr)
=c
$$
for $ds$-a.e. point of $\Gamma$. Thus
\begin{equation}\label{eq:smirnov-normal-u}
\partial_{\nu_{\rm in}}u=c
\qquad\text{for }ds\text{-a.e. on }\Gamma.
\end{equation}

\medskip

\noindent{\bf Step 2: A maximum principle for $|Du|$.} Now we are ready to follow the argument of Weinberger \cite{W1971}.
Let us pull everything back to the unit disk. Define
$$
v:=u\circ\varphi,\qquad
\widetilde F:=F\circ\varphi,\qquad
W:=\widetilde F-\overline\varphi.
$$
By \eqref{eq:smirnov-F-bounded}, $\widetilde F\in H^\infty(\mathbb D)$, so $W$ is bounded
and harmonic in $\mathbb D$. By Fatou's theorem for bounded analytic functions
\cite[Theorem~1.3]{D1970}, $\widetilde F$ has a.e. nontangential boundary values on
$\mathbb T$, and \eqref{eq:smirnov-F-boundary} yields
\begin{equation}\label{nontangential W}
  W(re^{it})\to W^\ast(e^{it}),
\quad  W^*(e^{it})=-2ic\,\overline{\tau(t)}
\qquad\text{for a.e. }t\in[0,2\pi].
\end{equation}
Hence
$$
|W^*|=2c
\qquad\text{a.e. on }\mathbb T.
$$

Since a bounded harmonic function is the Poisson integral of its boundary values,
$$
|W(w)|
=
\left|\frac{1}{2\pi}\int_0^{2\pi} P_w(e^{it})\,W^*(e^{it})\,dt\right|
\le
\frac{1}{2\pi}\int_0^{2\pi} P_w(e^{it})\,|W^*(e^{it})|\,dt
\le 2c
$$
for every $w\in\mathbb D$.
Furthermore, by the chain rule and \eqref{eq:smirnov-dz-u},
\begin{equation}\label{derivative of v}
4\partial_w v(w)
= \bigl(4\partial_z u\bigr)(\varphi(w))\,\varphi'(w)
= \bigl(F(\varphi(w))-\overline{\varphi(w)}\bigr)\varphi'(w)
= W(w)\varphi'(w).
\end{equation}
Therefore
\begin{equation}\label{eq:smirnov-grad-bound}
| D u(\varphi(w))|
= \frac{| D v(w)|}{|\varphi'(w)|}
= \frac{2|\partial_w v(w)|}{|\varphi'(w)|}
= \frac{|W(w)|}{2}
\le c
\qquad\text{for every }w\in\mathbb D.
\end{equation}
Moreover,   by \eqref{derivative of v} with $W\in L^\infty(\mathbb D)$, we have $Dv\in L^2(\mathbb D)$. 
Then as $v$ is continuous on $\overline{\mathbb D}$
and $v=0$ on $\mathbb T$, we have $v\in W^{1,2}_0(\mathbb D)$. By the conformal
invariance of the Dirichlet integral, $u\in W^{1,2}_0(\Omega)$.
In particular,
$$
u\in W^{1,\infty}(\Omega)\cap W^{1,2}_0(\Omega).
$$
By \eqref{eq:smirnov-Delta-u}, \eqref{eq:smirnov-u-zero-boundary}, and the maximum
principle,
\begin{equation}\label{eq:smirnov-u-positive}
u>0 \qquad\text{in }\Omega.
\end{equation}

We next show the boundary convergence for the radial derivative of
$v=u\circ\varphi$. Recall that
$$
4\partial_z u=F-\bar z, \quad
\widetilde F:=F\circ\varphi,
        \quad
W(w):=\widetilde F(w)-\overline{\varphi(w)},
$$
together with \eqref{nontangential W} and \eqref{derivative of v}.
Since $\varphi'\in H^1(\mathbb D)$, we have
$$
\varphi'(re^{it})\to (\varphi')^\ast(e^{it})
\quad\text{in }L^1(\mathbb T)
$$
as $r\to1$. Also, as $W$ is bounded in the sense $\sup_{0<r<1}\|W(re^{it})\|_{L^\infty(\mathbb T)}<\infty$,
\eqref{nontangential W} implies
$$
W(re^{it})\varphi'(re^{it})
\to 
W^\ast(e^{it})(\varphi')^\ast(e^{it})
\quad\text{in }L^1(\mathbb T);
$$
indeed, as $r\to 1$,
\begin{align*}
&\left\|
W(re^{it})\varphi'(re^{it})
-W^\ast(e^{it})(\varphi')^\ast(e^{it})
\right\|_{L^1(\mathbb T)}\\
\le & \ \|W(re^{it})\|_{L^\infty(\mathbb T)}
\left\|\varphi'(re^{it})-(\varphi')^\ast(e^{it})
\right\|_{L^1(\mathbb T)} \\
&\quad \quad+ \left\|\bigl(W(re^{it})-W^\ast(e^{it})\bigr)(\varphi')^\ast(e^{it})\right\|_{L^1(\mathbb T)}
\to 0.
\end{align*}
Now for real-valued $v$,
$$
\partial_r v(re^{it})
=
2\operatorname{Re}\left(e^{it}\partial_w v(re^{it})\right)
=
\frac12
\operatorname{Re}
\left(
e^{it}W(re^{it})\varphi'(re^{it})
\right).
$$
Hence
$$
\partial_r v(re^{it})
\to 
\frac12
\operatorname{Re}
\left(
e^{it}W^\ast(e^{it})(\varphi')^\ast(e^{it})
\right)
\quad\text{in }L^1(\mathbb T).
$$
Finally, since
$$
\tau(t)
=
\frac{z_t(t)}{|z_t(t)|}
=
\frac{i e^{it}(\varphi')^\ast(e^{it})}
     {|(\varphi')^\ast(e^{it})|},
$$
we have
$$
\overline{\tau(t)}
=
\frac{-i e^{-it}\overline{(\varphi')^\ast(e^{it})}}
     {|(\varphi')^\ast(e^{it})|}.
$$
Using \eqref{nontangential W} again, we obtain
$$
\begin{aligned}
\frac12
\operatorname{Re}
\left(
e^{it}W^\ast(e^{it})(\varphi')^\ast(e^{it})
\right)
&=
\frac12
\operatorname{Re}
\left(
e^{it}(-2ic\,\overline{\tau(t)})
(\varphi')^\ast(e^{it})
\right)
\\
&=
\frac12
\operatorname{Re}
\left(
-2c\,|(\varphi')^\ast(e^{it})|
\right)
\\
&=
-c\,|(\varphi')^\ast(e^{it})|.
\end{aligned}
$$
Consequently,
\begin{equation}\label{boundary convergence c}
\partial_r v(re^{it})
\to 
-c\,|(\varphi')^\ast(e^{it})|
\quad\text{in }L^1(\mathbb T)
\end{equation}
as $r\to1$.

\medskip

\noindent{\bf Step 3: The volume identity.}
We next pull back \eqref{mean value Serrin} to the unit disk.
If $H$ is harmonic in $\mathbb D$ and continuous on $\overline{\mathbb D}$, then
$h:=H\circ\varphi^{-1}$ is harmonic in $\Omega$ and continuous on $\overline\Omega$.
Applying \eqref{mean value Serrin} to $h$ and changing variables, we obtain
\begin{equation}\label{eq:smirnov-weighted-disk}
\int_{\mathbb D} H(w)\,|\varphi'(w)|^2\,dA(w)
=
c\int_0^{2\pi} H(e^{it})\,|(\varphi')^*(e^{it})|\,dt.
\end{equation}

Set
$$
I:=\int_\Omega u\,dA=\int_{\mathbb D} v(w)\,|\varphi'(w)|^2\,dA(w),
$$
$$
M:=\int_\Omega |z|^2\,dA(z)
=
\int_{\mathbb D} |\varphi(w)|^2\,|\varphi'(w)|^2\,dA(w),
$$
and
$$
B:=\int_\Gamma |z|^2\,ds
=
\int_0^{2\pi} |\varphi(e^{it})|^2\,|(\varphi')^*(e^{it})|\,dt.
$$

Let
$$
g(w):=|\varphi(w)|^2.
$$
Since $\varphi$ is holomorphic,
$$
\Delta g = 4|\varphi'|^2
\qquad\text{in }\mathbb D.
$$
Since $v=0$ on $\mathbb T$, $\varphi\in H^\infty$, one can pass $r\to 1$ in Green's formula \eqref{green} and get
\begin{equation}\label{DvDg}
    \int_{\mathbb D} Dv\cdot Dg\,dA=-4\int_{\mathbb D} v|\varphi'|^2\,dA =-4I.
\end{equation} 
Moreover, since $\partial_rv\to -c|(\varphi')^*|$ in $L^1(\mathbb T)$ according to \eqref{boundary convergence c},
$$\int_{\mathbb D} Dv\cdot Dg\,dA = \int_{\mathbb D} g|\varphi'|^2\, dA - c \int_0^{2\pi}|\varphi(e^{it})|^2|(\varphi')^*(e^{it})| \,dt =M-cB.$$
Thus,
\begin{equation}\label{eq:smirnov-first-identity}
4I+M=cB.
\end{equation}

Next define the bounded harmonic function
$$
J(w):=\operatorname{Re}\bigl(\varphi(w)\widetilde F(w)\bigr).
$$
Since $J$ need not extend continuously to $\overline{\mathbb D}$, we approximate it by
$$
\widetilde F_r(w):=\widetilde F(rw),
\qquad
J_r(w):=\operatorname{Re}\bigl(\varphi(w)\widetilde F_r(w)\bigr),
\qquad 0<r<1.
$$
Each $J_r$ is harmonic in $\mathbb D$ and continuous on $\overline{\mathbb D}$, so
\eqref{eq:smirnov-weighted-disk} applies to $J_r$. Since
$\widetilde F\in H^\infty(\mathbb D)$, the family $\{J_r\}_{0<r<1}$ is uniformly bounded,
$J_r\to J$ locally uniformly in $\mathbb D$, and
$J_r(e^{it})\to J^*(e^{it})$ for a.e. $t$. Hence dominated convergence yields
\begin{equation}\label{eq:smirnov-weighted-J}
\int_{\mathbb D} J(w)\,|\varphi'(w)|^2\,dA(w)
=
c\int_0^{2\pi} J^*(e^{it})\,|(\varphi')^*(e^{it})|\,dt.
\end{equation}

We first compute the boundary term in \eqref{eq:smirnov-weighted-J}. On $\mathbb T$,
$$
\widetilde F^*(e^{it})
=
\overline{\varphi(e^{it})}-2ic\,\overline{\tau(t)},
$$
hence
$$
J^*(e^{it})
=
|\varphi(e^{it})|^2+\operatorname{Re}\bigl(-2ic\,\varphi(e^{it})\overline{\tau(t)}\bigr).
$$
If we identify $z=x_1+ix_2$ with $x=(x_1,x_2)\in\mathbb R^2$, then
$x\cdot\nu_{\rm out}=\operatorname{Re}(z\overline\nu_{\rm out})$, and since $\nu_{\rm out}=-i\tau$,
$$
\operatorname{Re}\bigl(-2ic\,z\,\overline\tau\bigr)
=
-2c\,x\cdot\nu_{\rm out}.
$$
Therefore
$$
c\int_0^{2\pi} J^*(e^{it})\,|(\varphi')^*(e^{it})|\,dt
=
cB-2c^2\int_\Gamma x\cdot\nu_{\rm out}\,ds.
$$
By the divergence theorem,
$$
\int_\Gamma x\cdot\nu_{\rm out}\,ds
=
\int_\Omega \operatorname{div}x\,dA
=
2|\Omega|.
$$
Hence
\begin{equation}\label{eq:smirnov-J-boundary}
\int_{\mathbb D} J(w)\,|\varphi'(w)|^2\,dA(w)
=
cB-4c^2|\Omega|.
\end{equation}

We next compute   the interior term \eqref{eq:smirnov-weighted-J}. Since
$\widetilde F=\overline\varphi+W$,
$$
J|\varphi'|^2
=
|\varphi|^2|\varphi'|^2+\operatorname{Re}(\varphi W)\,|\varphi'|^2.
$$
Also,
$$
\partial_w g = \varphi'(w)\,\overline{\varphi(w)},
\qquad
\partial_{\overline w}g = \overline{\varphi'(w)}\,\varphi(w),
$$
and, since $W\varphi'=4\partial_w v$,
$$
 D v\cdot D g
=
4\operatorname{Re}\bigl(\partial_w v\,\partial_{\overline w}g\bigr)
=
4\operatorname{Re}\bigl(\partial_w v\,\overline{\varphi'}\,\varphi\bigr)
=
\operatorname{Re}(\varphi W)\,|\varphi'|^2.
$$
Thus
$$
\int_{\mathbb D} J|\varphi'|^2\,dA
=
M+\int_{\mathbb D} D v\cdot D g\,dA.
$$
Therefore, \eqref{DvDg} implies
\begin{equation}\label{eq:smirnov-J-interior}
\int_{\mathbb D} J(w)\,|\varphi'(w)|^2\,dA(w)=M-4I.
\end{equation}

Comparing \eqref{eq:smirnov-J-boundary} and \eqref{eq:smirnov-J-interior}, we obtain
\begin{equation}\label{eq:smirnov-second-identity}
M-4I=cB-4c^2|\Omega|.
\end{equation}
Subtracting \eqref{eq:smirnov-second-identity} from
\eqref{eq:smirnov-first-identity} gives
\begin{equation}\label{eq:smirnov-volume-identity}
2\int_\Omega u\,dA = c^2|\Omega|.
\end{equation}

\medskip
\noindent{\bf Step 4: Weinberger's $P$-function argument.}
Due to \eqref{eq:smirnov-volume-identity} and \eqref{eq:smirnov-grad-bound}, we conclude the theorem by the argument of Weinberger's $P$-function argument \cite{W1971}; the details are omitted. 
\end{proof}

\section{Non-Smirnov  case in the plane}\label{sec:non-smirnov-case}

We now turn to the non-Smirnov side and use the Duren--Shapiro--Shields (DSS) construction to show that the strong quadrature identity \eqref{mean value Serrin} behaves very differently from the one-pole pseudosphere condition of Lewis--Vogel.

The  proposition below rules out the most naive DSS-class, namely a fixed purely singular branch, which was considered by \cite{KL1937}. This means that one cannot expect simple non-Smirnov domains satisfying \eqref{mean value Serrin}.  

\begin{prop} \label{prop:pure-singular-no-go}
Let $\mu$ be a nonzero positive singular measure of class $A^*$ on $\mathbb T$, and define
$$
F_\mu(z):=\mathcal H[\mu](z),
\qquad
f_a(z):=\int_0^z e^{-aF_\mu(\xi)}\,d\xi,
\qquad a>0.
$$
Suppose that there exists a sequence $a_j\to  0$ such that each $f_{a_j}$ is conformal in $\mathbb D$, extends homeomorphically to $\overline{\mathbb D}$, and maps $\mathbb D$ onto a rectifiable Jordan domain $\Omega_{a_j}$. Assume  that for every $j$ there exists $c_j>0$ such that
\begin{equation}\label{eq:pure-singular-quadrature}
\int_{\Omega_{a_j}} h\,dA
=
c_j\int_{\partial\Omega_{a_j}} h\,ds
\qquad\text{for every }h\in C(\overline{\Omega_{a_j}}),\ \Delta h=0.
\end{equation}
Then $\mu=0$. In particular, no nontrivial fixed pure-singular DSS branch can produce a perturbative family of domains satisfying \eqref{mean value Serrin}.
\end{prop}

\begin{proof}
For each $j$ set $f_j:=f_{a_j}$ and $\Omega_j:=\Omega_{a_j}$. Since $a_j\mu$ is purely singular, \eqref{eq:boundary-modulus-Hnu} gives
\begin{equation}\label{eq:pure-singular-boundary-modulus}
|f_j'^*(\zeta)|=1
\qquad\text{for }m\text{-a.e. }\zeta\in\mathbb T.
\end{equation}
Let $H\in C(\overline{\mathbb D})$ be harmonic in $\mathbb D$ and set $h:=H\circ f_j^{-1}$. Then $h\in C(\overline{\Omega_j})$ is harmonic in $\Omega_j$. Changing variables in \eqref{eq:pure-singular-quadrature} and using \eqref{eq:pure-singular-boundary-modulus}, we obtain
\begin{equation}\label{eq:pure-singular-pullback}
\int_{\mathbb D} H(z)e^{-2a_jP[\mu](z)}\,da(z)
=
2{c_j} \int_{\mathbb T} H(\zeta)\,dm(\zeta)
\end{equation}
for every harmonic $H$ on $\mathbb D$.

Fix $n\ge 1$ and take $H(z)=z^n$. Since $\int_{\mathbb T}\zeta^n\,dm(\zeta)=0$, \eqref{eq:pure-singular-pullback} yields
$$
M_n(a_j):=\int_{\mathbb D} z^n e^{-2a_jP[\mu](z)}\,da(z)=0
\qquad\text{for every }j.
$$
As $P[\mu]\in L^1(\mathbb D,da)$ and
$$
\frac{1-e^{-2aP[\mu](z)}}{a}\to 2P[\mu](z)
\qquad\text{for a.e. }z\in\mathbb D,
$$
dominated convergence gives
$$
M_n'(0)=-2\int_{\mathbb D} z^n P[\mu](z)\,da(z).
$$
Since $M_n(a_j)=0$ along a sequence converging to $0$ and
$$M_n(0)=\int_{\mathbb D}z^n\,da(z)=0,$$
we get $M_n'(0)=0$ and therefore
\begin{equation}\label{eq:pure-singular-moment-zero}
\int_{\mathbb D} z^n P[\mu](z)\,da(z)=0
\qquad\text{for every }n\ge 1.
\end{equation}
Using \eqref{eq:Herglotz-real-part}, Fubini's theorem, and the identity
$$
\int_{\mathbb D} z^n P_z(\zeta)\,da(z)=\frac{\zeta^n}{n+1}
\qquad\text{for }\zeta\in\mathbb T,
$$
we obtain from \eqref{eq:pure-singular-moment-zero}
$$
0
=
\int_{\mathbb T}\left(\int_{\mathbb D} z^n P_z(\zeta)\,da(z)\right)d\mu(\zeta)
=
\frac{1}{n+1}\int_{\mathbb T}\zeta^n\,d\mu(\zeta)
\qquad (n\ge 1).
$$
Thus all positive Fourier coefficients of $\mu$ vanish. Since $\mu$ is a real finite measure, all negative Fourier coefficients also vanish. Therefore $\mu$ is a constant multiple of Lebesgue measure. Since $\mu$ is singular, this constant must be zero. Hence $\mu=0$, a contradiction.
\end{proof}

\begin{rem}\label{difference}
Proposition~\ref{prop:pure-singular-no-go} indicates that,  the equation  \eqref{pseudoballs}  for pseudoballs is very difference from  the strong quadrature identity \eqref{mean value Serrin}. Indeed, pseudoballs are compatible with a pure singular branch, while the strong quadrature identity is not. 

Moreover, observe that disk has a $2$-dimensional family of  M\"obius-center changes, which means the linearized quadrature equation around the disk has an extra finite-dimensional kernel generated by the infinitesimal symmetry group of the disk. This also makes a big difference from the pseudoball case. Then, instead of performing a Lyapunov--Schmidt reduction to quotient out those symmetry directions, we impose a symmetry class below so that those directions are forbidden. This gives us the desired counterexample below.
\end{rem}

\begin{proof}[Proof of Theorem~\ref{thm:counterexample-branch}]
Let $f\in C^2((-\rho,\rho))$ and set $f_0:=f(0)>0$.  Choose a nonzero positive singular measure $\sigma$ of class $A^*$ on $\mathbb T$.  Let
$$
    \mu:=\frac14\sum_{k=0}^3 (R_{\pi k/2})_\#\sigma,
    \qquad R_{\pi k/2}(\zeta)=i^k\zeta.
$$
Then $\mu$ is nonzero, positive, singular, of class $A^*$, and four-fold symmetric.  In particular,
\begin{equation}\label{eq:mu-4fold}
    \widehat\mu(n)=0
    \qquad\hbox{whenever }4\nmid n.
\end{equation}
Set $M_\mu:=P[\mu]$.

Fix $0<\alpha<1$ and use the Banach space $X_\alpha^4$ defined in \eqref{eq:Xalpha4}.  For $W\in X_\alpha^4$ and $a\ge0$, define
\begin{equation}\label{eq:GWa-def}
    G_{W,a}(z):=\exp\{-2(P[W](z)+aM_\mu(z))\},
    \qquad z\in\mathbb D.
\end{equation}
For $\lambda>0$ small, let $v_{W,a,\lambda}$ be the unique solution of
\eqref{eq:semilinear-disk-dirichlet}, and define the semilinear density
\begin{equation}\label{eq:semilinear-density-def}
    G^{f,\lambda}_{W,a}(z)
    :=G_{W,a}(z)f(\lambda^2v_{W,a,\lambda}(z)).
\end{equation}
We define the nonlinear operator
\begin{equation}\label{eq:Psi-def}
    \Psi_{f,\lambda}(W,a)
    :=
    \log \ct(G^{f,\lambda}_{W,a})+W
    -\int_{\mathbb T}\bigl(\log \ct(G^{f,\lambda}_{W,a})+W\bigr)\,dm.
\end{equation}
The logarithm is well defined for $W$, $a$, and $\lambda$ small.  Indeed, since $M_\mu\ge0$ and $W$ is bounded, $G_{W,a}$ is bounded from above; on $|z|\le1/2$ it is bounded from below by a positive constant depending only on $\|W\|_{X^4_\az}$ and $a$.  Since $f(\lambda^2v_{W,a,\lambda})=f_0+O(\lambda^2)$ by \eqref{eq:semilinear-source-small}, the same lower bound holds for $G^{f,\lambda}_{W,a}$ when $\lambda$ is small.  Hence $\ct(G^{f,\lambda}_{W,a})$ is bounded below on $\mathbb T$.

\medskip
\noindent{\bf Step 1: Regularity of the nonlinear operator.}
We first compare \eqref{eq:Psi-def} with the constant-source operator
$$
    \Psi(W,a):=\log \ct(G_{W,a})+W
    -\int_{\mathbb T}\bigl(\log \ct(G_{W,a})+W\bigr)\,dm.
$$
Since $f(\lambda^2v_{W,a,\lambda})=f_0+O(\lambda^2)$ uniformly in $L^\infty(\mathbb D)$, Lemma~\ref{lem:T-holder} gives
$$
    \ct(G^{f,\lambda}_{W,a})
    =f_0\ct(G_{W,a})+O(\lambda^2)
    \qquad\hbox{in }C^\alpha(\mathbb T).
$$
The positive constant $f_0$ disappears after subtracting the boundary mean of the logarithm.  Therefore
\begin{equation}\label{eq:Psi-semi-perturb}
    \Psi_{f,\lambda}(W,a)=\Psi(W,a)+O(\lambda^2)
    \qquad\hbox{in }X_\alpha^4,
\end{equation}
uniformly for $W$ and $a$ in a sufficiently small set.  Moreover, by
\eqref{eq:semilinear-source-derivative-small},
\begin{equation}\label{eq:DPsi-semi-perturb}
    D_W\Psi_{f,\lambda}(W,a)=D_W\Psi(W,a)+O(\lambda^2)
    \qquad\hbox{in }\mathcal L(X_\alpha^4,X_\alpha^4).
\end{equation}

\medskip
\noindent{\bf Step 2: The linearized operator.}
We now recall the linearized operator at the disk.  At $(W,a)=(0,0)$,
$$
    D_W\Psi(0,0)=I-2K,
    \qquad K:=\ct\circ P.
$$
For the Fourier mode ${\bf e}_k(\zeta)=\zeta^k$ one has
\begin{align*}
    K({\bf e}_k)(e^{i\varphi})=& \ \int_0^1\int_{\mathbb T} \left(\sum_{\ell\in \mathbb Z} r^{|\ell|} e^{i\ell(\theta-\varphi)} \right)r^{|k|}e^{i k \theta} \,dm(\theta) 2r\,dr\\
    = &\ 2e^{i k \varphi}\int_0^1 r^{|k|+1}\, dr     =\frac{1}{|k|+1} e^{ik\varphi}   =\frac{1}{|k|+1}{\bf e}_k(e^{i\varphi}).
\end{align*}
and hence
$$
    K({\bf e}_k)=\frac1{|k|+1}{\bf e}_k.
$$
Since functions in $X_\alpha^4$ have only Fourier modes $n\in4\mathbb Z\setminus\{0\}$, all eigenvalues of $I-2K$ on $X_\alpha^4$ are
$$
    \frac{|k|-1}{|k|+1}>0.
$$
As in the proof of the constant-source construction, $K$ is compact on $X_\alpha^4$  since
$$
K: X_\alpha^4 \hookrightarrow L^\infty(\mathbb T)
\xrightarrow{\ P\ }
L^\infty(\mathbb D)
\xrightarrow{\ \ct
 }
C^\beta(\mathbb T)
\hookrightarrow X_\alpha^4
$$
for every $\beta\in(\alpha,1)$, where the last embedding is compact. Hence $I-2K$ is invertible on $X_\alpha^4$ by the Fredholm alternative.
Thus
\begin{equation}\label{eq:linearized-invertible}
    A:=I-2K:X_\alpha^4\to X_\alpha^4
\end{equation}
is an isomorphism.

\medskip
\noindent{\bf Step 3: Contraction.}
Now we run Banach fixed point theorem to find a solution $(W(a), a)$ to 
$$\Psi_{f,\,\lambda}(W,\,a)=0.$$

 Choose $\epsilon_0>0$ and then choose $a_0>0$ and $\lambda_0>0$ so small that
\begin{equation}\label{eq:contraction-derivative-bound}
    \|A^{-1}(D_W\Psi_{f,\lambda}(W,a)-A)\|_{\mathcal L(X_\alpha^4)}
    \le \frac12
\end{equation}
whenever $\|W\|_{X_\alpha^4}\le\epsilon_0$, $0\le a\le a_0$, and $0<\lambda<\lambda_0$.  This follows from the continuity of $D_W\Psi(W,a)$ near $(0,0)$ and from \eqref{eq:DPsi-semi-perturb}.  Shrinking $a_0$ and $\lambda_0$  further if necessary, we also have
\begin{equation}\label{eq:contraction-center-bound}
    \|A^{-1}\Psi_{f,\lambda}(0,a)\|_{X_\alpha^4}
    \le \frac{\epsilon_0}{2}.
\end{equation}
For fixed $a$ and $\lambda$ define
$$
    \Gamma_{a,\lambda}(W):=W-A^{-1}\Psi_{f,\lambda}(W,a).
$$
By the fundamental theorem of calculus in the Banach space $X_\alpha^4$, equations \eqref{eq:contraction-derivative-bound} and \eqref{eq:contraction-center-bound} imply that $\Gamma_{a,\lambda}$ maps the closed ball
$\overline B_{\epsilon_0}\subset X_\alpha^4$ into itself and is a contraction with Lipschitz constant at most $1/2$.  Hence there is a unique
$W(a,\lambda)\in\overline B_{\epsilon_0}$ such that
\begin{equation}\label{eq:Psi-zero}
    \Psi_{f,\lambda}(W(a,\lambda),a)=0.
\end{equation}
Since $X_\alpha^4$ consists of mean-zero four-fold symmetric functions, \eqref{eq:Psi-zero} is equivalent to the existence of a positive constant $C(a,\lambda)$ such that
\begin{equation}\label{eq:fixed-point-equation}
    \ct(G^{f,\lambda}_{W(a,\lambda),a})
    =C(a,\lambda)e^{-W(a,\lambda)}
    \qquad\hbox{on }\mathbb T.
\end{equation}

\medskip
\noindent{\bf Step 4: Build the DSS datum.}
 Define the real Borel measure
$$
    d\nu_{a,\lambda}:=W(a,\lambda)\,dm+a\,d\mu.
$$
The absolutely continuous part $W(a,\lambda)dm$ is of class $A^*$, and the singular part is $a\mu$.  For $a$ and $\lambda$ sufficiently small, Lemma~\ref{lem:small-data-DSS} applies.  Let
$$
    F_{a,\lambda}:=\mathcal H[\nu_{a,\lambda}],
    \qquad
    f_{a,\lambda}(z):=\int_0^z e^{-F_{a,\lambda}(\xi)}\,d\xi,
    \qquad
    \Omega_{a,\lambda}^{\rm ref}:=f_{a,\lambda}(\mathbb D).
$$
Then $f_{a,\lambda}$ is conformal in $\mathbb D$, extends homeomorphically to $\overline{\mathbb D}$, and maps $\mathbb D$ onto a rectifiable Jordan domain.  Since the singular part is $a\mu\ne0$ for $a>0$, the image domain is not Smirnov.

By \eqref{eq:Herglotz-real-part} and \eqref{eq:boundary-modulus-Hnu},
\begin{equation}\label{interior Jacobian density}
    |f_{a,\lambda}'(z)|^2
    =G_{W(a,\lambda),a}(z),
    \qquad z\in\mathbb D,
\end{equation}
and
$$
    |f_{a,\lambda}'{}^*(\zeta)|=e^{-W(a,\lambda)(\zeta)}
    \qquad\hbox{for }m\hbox{-a.e. }\zeta\in\mathbb T.
$$
Thus \eqref{eq:fixed-point-equation} becomes
\begin{equation}\label{eq:fixed-point-equation-fa}
    \ct\left(|f_{a,\lambda}'|^2
    f(\lambda^2v_{W(a,\lambda),a,\lambda})\right)
    =C(a,\lambda)|f_{a,\lambda}'{}^*|
    \qquad\hbox{on }\mathbb T.
\end{equation}

\medskip
\noindent{\bf Step 5: Conclusion.}
Let $\Omega_{a,\lambda}:=\lambda\Omega_{a,\lambda}^{\rm ref}$ be the physical domain.  Define $u_{a,\lambda}$ on $\Omega_{a,\lambda}$ by
\begin{equation}\label{eq:physical-scaling}
    u_{a,\lambda}(\lambda f_{a,\lambda}(z))
    :=\lambda^2 v_{W(a,\lambda),a,\lambda}(z).
\end{equation}
Then $u_{a,\lambda}=0$ on $\partial\Omega_{a,\lambda}$ and, by the conformal invariance of the Laplacian in two dimensions,
$$
    -\Delta u_{a,\lambda}=f(u_{a,\lambda})
    \qquad\hbox{in }\Omega_{a,\lambda}.
$$

We next prove the harmonic moment identity.  Let $H\in C(\overline{\mathbb D})$ be harmonic in $\mathbb D$, and let $h:=H\circ f_{a,\lambda}^{-1}$ on $\Omega_{a,\lambda}^{\rm ref}$.  By \eqref{eq:T-fubini} and \eqref{eq:fixed-point-equation-fa},
$$
\begin{aligned}
    &\int_{\mathbb D}H(z)|f_{a,\lambda}'(z)|^2
    f(\lambda^2v_{W(a,\lambda),a,\lambda}(z))\,da(z) \\
    &\qquad =
    C(a,\lambda)\int_{\mathbb T}H(\zeta)|f_{a,\lambda}'{}^*(\zeta)|\,dm(\zeta).
\end{aligned}
$$
Changing variables, using $da=dA/\pi$ and $ds=2\pi|f_{a,\lambda}'{}^*|dm$, gives
$$
    \int_{\Omega_{a,\lambda}^{\rm ref}}f(\lambda^2v_{W(a,\lambda),a,\lambda})h\,dA
    =\frac{C(a,\lambda)}2\int_{\partial\Omega_{a,\lambda}^{\rm ref}}h\,ds.
$$
After scaling by $\lambda$, this becomes
\begin{equation}\label{eq:semilinear-domain-moment}
    \int_{\Omega_{a,\lambda}} f(u_{a,\lambda})\,\psi\,dA
    =c(a,\lambda)\int_{\partial\Omega_{a,\lambda}}\psi\,ds
\end{equation}
for every harmonic $\psi\in C(\overline{\Omega_{a,\lambda}})$, where
$$
    c(a,\lambda):=\lambda\frac{C(a,\lambda)}2.
$$
Indeed, if $\psi$ is harmonic on the physical domain, then $\psi(\lambda\cdot)$ is harmonic on the reference domain, volume measure scales by $\lambda^2$, and arclength by $\lambda$.

Finally, Lemma~\ref{lem:quadrature-to-weak-serrin} with $q=f(u_{a,\lambda})$ and the moment identity \eqref{eq:semilinear-domain-moment} gives
$$
    \Delta u_{a,\lambda}
    =c(a,\lambda)\mathcal H^1\!\lfloor_{\partial\Omega_{a,\lambda}}
    -f(u_{a,\lambda})\mathbf 1_{\Omega_{a,\lambda}}\,dA
    \qquad\hbox{in }\mathcal D'(\mathbb C).
$$
The zero extension of $u_{a,\lambda}$ is compactly supported and vanishes on
$\mathbb C\setminus\overline{\Omega_{a,\lambda}}$.
The domain is not Smirnov since the singular part of $\nu_{a,\lambda}$ is $a\mu\ne0$.  It is not a disk, since disks are Smirnov domains.  The four-fold symmetry of $\mu$ and $W(a,\lambda)$ gives the desired rotational symmetry.  The proof is complete.
\end{proof}

\section{Higher-dimensional semilinear counterexamples}\label{sec:higher-dimensional-case}

In this section, we establish Theorem~\ref{thm:higher-dimensional-counterexample}. The overall strategy is analogous to that of Theorem~\ref{thm:counterexample-branch}, but the details are significantly more delicate. The proof relies on three additional components: a suitable symmetry assumption on the singular data, a weighted elliptic trace estimate on the disk, and a quantitative control of the singular perturbation. Taken together, these ensure that the first Fourier mode generated by the weight can be eliminated while preserving the positivity of the singular measure, similar to the four-fold symmetry used in the planar setting.

Throughout this section, set $n\ge3$ and
$ p:=n-2.$
Let
$$
A(\zeta):=-\zeta, \qquad S(\zeta):=-\overline\zeta,
    \qquad \zeta\in\mathbb T.
$$
Here $A$ is the antipodal map and $S$ is the reflection across the imaginary axis.  

Let $0<\eta<1/10$ be a small parameter, and set
$$
I_+:=\{e^{it}: |t-\pi/2|<\eta\},
\qquad I_-:=A(I_+).
$$
Choose a nonzero positive singular probability measure $\sigma_+$ of class $A^*$, supported in $I_+$, such that
$$
    S_\#\sigma_+=\sigma_+.
$$
Such a measure can be obtained by taking a positive singular $A^*$ measure on a small interval and symmetrizing it via $S$.  Define
\begin{equation}\label{eq:defn-nu-sigma}
 \sigma_-:=A_\#\sigma_+,
\qquad  \Sigma:=\sigma_++\sigma_-,
 \qquad \nu:=\sigma_- -\sigma_+.    
\end{equation}
Then $\Sigma$ is a finite positive singular $A^\ast$-measure with $\Sigma(\mathbb T)=2,$
and $\nu$ is a finite signed $A^\ast$-measure satisfying
$$
|\nu|\le \Sigma,\qquad |P[\nu]|\le P[\Sigma].
$$

Moreover,
$$
A_\#\Sigma=\Sigma,
\qquad S_\#\Sigma=\Sigma,
\qquad A_\#\nu=-\nu,
\qquad S_\#\nu=\nu.
$$
The exact antipodal symmetry of $\Sigma$ will be used to remove the  sine mode.
For $a>0$ and $|\delta|<a$, set
\begin{equation}\label{eq:hd-singular-measure}
    \mu_{a,\delta}:=a\Sigma+\delta\nu
    =(a-\delta)\sigma_+ +(a+\delta)\sigma_- .
\end{equation}
Then $\mu_{a,\delta}$ is a nonnegative singular $A^*$ measure whenever $|\delta|<a$.

Fix a small H\"older exponent
\begin{equation}\label{eq:alpha-choice-hd}
    0<\alpha<\frac14.
\end{equation}
Let $\mathcal X_\alpha$ be the real Banach space of all $W\in C^\alpha(\mathbb T)$ such that
$$
\int_{\mathbb T}W\,dm=0,
 \qquad W(-\overline\zeta)=W(\zeta).
$$
Note that, the only first Fourier mode compatible with this reflection symmetry is $\sin t$. We denote by
$$
 \pi_{\sin}g
:= 2\int_{\mathbb T} g(e^{it})\sin t\,dm(e^{it})
$$
the scalar coefficient of the $\sin t$-mode, and set
$$
\mathcal P_{\sin}g:=(\pi_{\sin}g)\sin t,\qquad
    \mathcal P_{\rm hi}g:=\mathcal P_0g-\mathcal P_{\sin}g,
$$
where $\mathcal P_0g=g-\int_{\mathbb T}g\,dm$ is the mean-zero projection.
We eventually define
$$
 \mathcal Y_\alpha:= \{W\in\mathcal X_\alpha:\pi_{\sin}W=0\}.
$$
Thus $\mathcal P_{\rm hi}$ projects onto the high-mode component in the reflected
mean-zero class, while $\pi_{\sin}$ is a scalar functional.

For $W\in\mathcal Y_\alpha$, recall the Herglotz transform $\mathcal H[\cdot]$ and define $F_{W,a,\delta}$ so that
\begin{equation}\label{eq:hd-F-prime-corrected}
    F'_{W,a,\delta}(z)
    :=\exp\{-\mathcal H[W\,dm+\mu_{a,\delta}](z)\},
    \qquad
    F_{W,a,\delta}(0)=0.
\end{equation}
By Lemma~\ref{lem:small-data-DSS}, if $a$, $|\delta|$, and $\|W\|_{C^\alpha}$ are sufficiently small, then $F_{W,a,\delta}$ is conformal in $\mathbb D$, extends homeomorphically to $\overline{\mathbb D}$, and has rectifiable Jordan image.  Since $\mu_{a,\delta}$ is nonzero for the parameters used below, that image is not a Smirnov domain.  Moreover,
\begin{equation}\label{eq:Fprime-upper-bound}
    |F'_{W,a,\delta}(z)|
    =\exp\{-P[W](z)-P[\mu_{a,\delta}](z)\}
    \le e^{\|W\|_{L^\infty(\mathbb T)}}.
\end{equation}
Thus these conformal maps are uniformly Lipschitz on $\mathbb D$ as long as $W$ stays in a fixed small ball.

\subsection{Weighted elliptic trace operators}

We first record the following stability result, which will be used  in a small neighborhood of the origin in $C^\alpha(\mathbb T)$.

\begin{lem}\label{lem:hd-herglotz-stability}
Let $0<\alpha<1$ and let $M>0$.  Let $\mu$ be a finite nonnegative Borel
measure on $\mathbb T$, and define
$$
    F'_{W,\mu}(z):=
    \exp\{-\mathcal H[W\,dm+\mu](z)\},
    \qquad F_{W,\mu}(0)=0,
$$
where $W\in C^\alpha(\mathbb T)$ is real-valued and
$\|W\|_{C^\alpha(\mathbb T)}\le M$.  Then the map
$W\mapsto F_{W,\mu}$ is Frechet differentiable as a map into
$W^{1,\infty}(\mathbb D)$, and for every $H\in C^\alpha(\mathbb T)$,
$$
    D_WF'_{W,\mu}[H]
    =-\mathcal H[H\,dm]F'_{W,\mu}.
$$
Moreover
\begin{equation}\label{eq:Herglotz-W-stability}
    \|D_WF'_{W,\mu}[H]\|_{L^\infty(\mathbb D)}
    +\|D_WF_{W,\mu}[H]\|_{W^{1,\infty}(\mathbb D)}
    \le C_{\alpha,M}\|H\|_{C^\alpha(\mathbb T)} .
\end{equation}
The constant is independent of the nonnegative measure $\mu$.
\end{lem}

\begin{proof}
For $s\in\mathbb R$,
$$
F'_{W+sH,\mu}
=\exp\{-\mathcal H[W\,dm+\mu]\}\exp\{-s\mathcal H[H\,dm]\}.
$$
Thus
$$
\frac{d}{ds}\bigg|_{s=0}F'_{W+sH,\mu}
=-\mathcal H[H\,dm]F'_{W,\mu}.
$$
The Hilbert transform is bounded on $C^\alpha(\mathbb T)$, and hence the
analytic Herglotz transform of the absolutely continuous measure $H\,dm$
satisfies
$$
    \|\mathcal H[H\,dm]\|_{L^\infty(\mathbb D)}
    \le C_\alpha\|H\|_{C^\alpha(\mathbb T)} .
$$
On the other hand,
$$
    |F'_{W,\mu}(z)|
    =\exp\{-P[W](z)-P[\mu](z)\}
    \le e^{\|W\|_{L^\infty(\mathbb T)}}\le e^M,
$$
since  $P[\mu]\ge0$.  This proves the $L^\infty$ bound for
$D_WF'_{W,\mu}[H]$.  Since
$$
D_WF_{W,\mu}[H](z)=\int_0^zD_WF'_{W,\mu}[H](\xi)\,d\xi,
$$
the same estimate gives the $W^{1,\infty}$ bound in
\eqref{eq:Herglotz-W-stability}.
\end{proof}

For $\tau\ge0$ and $F=F_{W,a,\delta}$, define
$$
    A_{\tau,F}(z):=(1+\tau\operatorname{Im}F(z))^p.
$$
Since $F$ is uniformly bounded and Lipschitz in the small parameter region, for
$\tau$ sufficiently small,
\begin{equation}\label{eq:A-tau-lipschitz-close}
    \frac12\le A_{\tau,F}\le 2,
    \qquad
    \|A_{\tau,F}-1\|_{W^{1,\infty}(\mathbb D)}\le C\tau .
\end{equation}
Moreover, Lemma~\ref{lem:hd-herglotz-stability} implies the outer-variable
coefficient estimate
\begin{equation}\label{eq:A-tau-W-derivative}
    \|D_WA_{\tau,F_{W,a,\delta}}[H]\|_{W^{1,\infty}(\mathbb D)}
    \le C\tau\|H\|_{C^\alpha(\mathbb T)} .
\end{equation}
Indeed,
$$
D_WA_{\tau,F}[H]
=p\tau(1+\tau\operatorname{Im}F)^{p-1}
   \operatorname{Im}D_WF[H],
$$
and one spatial derivative of this expression contains only $F'$, $D_WF[H]$,
and $D_WF'[H]$, all of which are uniformly controlled by
Lemma~\ref{lem:hd-herglotz-stability} in the small parameter region.

For $\rho\in L^\infty(\mathbb D)$, let $\psi$ solve
\begin{equation}\label{eq:rho-equation}
     -\operatorname{div}(A_{\tau,F}\nabla\psi)=\rho
  \quad\text{in }\mathbb D,
  \qquad  \psi=0\quad\text{on }\mathbb T.
\end{equation}
Set
$$
\mathcal T_{\tau,F}\rho:=-2A_{\tau,F}\partial_\nu\psi\big|_{\mathbb T},
$$
where $\nu$ is the outward unit normal to $\mathbb T$, and the constant $2$ comes
from the normalization of measures.  Equivalently, if $U_h$ solves
$$
 \operatorname{div}(A_{\tau,F}\nabla U_h)=0
 \quad\text{in }\mathbb D, \qquad U_h=h\quad\text{on }\mathbb T,
$$
then Green's identity gives
\begin{equation}\label{eq:adjoint-trace-variable}
 \int_{\mathbb D}U_h\rho\,da
 =\int_{\mathbb T}h\,\mathcal T_{\tau,F}\rho\,dm.
\end{equation}
In particular, when $\tau=0$ and $A_{\tau,F}\equiv1$, one has
$\mathcal T_{0,F}=\mathcal T$; see \eqref{eq:T-fubini}.

\begin{lem}\label{lem:weighted-trace-stability}
Choose $q>2$ such that $\alpha<1-2/q$.  Let
$A\in W^{1,\infty}(\mathbb D)$ satisfy
$$
    \frac12\le A\le2,
    \qquad \|A\|_{W^{1,\infty}(\mathbb D)}\le M.
$$
For $\rho\in L^\infty(\mathbb D)$, let $\psi_A\in W^{1,2}_0(\mathbb D)$ solve
$$
    -\operatorname{div}(A\nabla\psi_A)=\rho
    \quad\text{in }\mathbb D,
$$
and set
$$
    \mathcal T_A\rho:=-2A\partial_\nu\psi_A|_{\mathbb T}.
$$
Then
\begin{equation}\label{eq:T-variable-holder}
    \|\mathcal T_A\rho\|_{C^\alpha(\mathbb T)}
    \le C\|\rho\|_{L^\infty(\mathbb D)} .
\end{equation}
If $A_1,A_2$ satisfy the same ellipticity and $W^{1,\infty}$ bounds, then
\begin{equation}\label{eq:T-coefficient-Lipschitz}
    \|\mathcal T_{A_1}\rho-\mathcal T_{A_2}\rho\|_{C^\alpha(\mathbb T)}
    \le
    C\|A_1-A_2\|_{W^{1,\infty}(\mathbb D)}
    \|\rho\|_{L^\infty(\mathbb D)} .
\end{equation}
In particular, for $A_{\tau,F}$ in \eqref{eq:A-tau-lipschitz-close},
\begin{equation}\label{eq:T-variable-perturbation}
    \|\mathcal T_{\tau,F}\rho-\ct\rho\|_{C^\alpha(\mathbb T)}
    \le C\tau\|\rho\|_{L^\infty(\mathbb D)} .
\end{equation}
\end{lem}

\begin{proof}
The equation can be written in non-divergence form as
$$
    -A\Delta\psi_A-\nabla A\cdot\nabla\psi_A=\rho .
$$
The disk is smooth and $A$ is uniformly elliptic with Lipschitz coefficients.
The global $W^{2,q}$ Dirichlet estimate for uniformly elliptic equations with
Lipschitz coefficients gives
$$
    \|\psi_A\|_{W^{2,q}(\mathbb D)}
    \le C\|\rho\|_{L^q(\mathbb D)}
    \le C\|\rho\|_{L^\infty(\mathbb D)} .
$$
Since $q>2$ and $\alpha<1-2/q$, the Sobolev trace embedding gives
$$
    \|\partial_\nu\psi_A\|_{C^\alpha(\mathbb T)}
    \le C\|\psi_A\|_{W^{2,q}(\mathbb D)}.
$$
As the trace of $A$ is Lipschitz on $\mathbb T$, this proves
\eqref{eq:T-variable-holder}.

Now let $\psi_j=\psi_{A_j}$, $j=1,2$, and put $w:=\psi_1-\psi_2$.  Subtracting
the equations gives
$$
    -\operatorname{div}(A_1\nabla w)
    =\operatorname{div}\bigl((A_1-A_2)\nabla\psi_2\bigr).
$$
Equivalently,
$$
-A_1\Delta w-\nabla A_1\cdot\nabla w
=(A_1-A_2)\Delta\psi_2+
  \nabla(A_1-A_2)\cdot\nabla\psi_2 .
$$
Using the already proved $W^{2,q}$ estimate for $\psi_2$,
$$
\|(A_1-A_2)\Delta\psi_2+
  \nabla(A_1-A_2)\cdot\nabla\psi_2\|_{L^q}
\le C\|A_1-A_2\|_{W^{1,\infty}}\|\rho\|_{L^\infty}.
$$
Applying the $W^{2,q}$ estimate to $w$ yields
$$
\|w\|_{W^{2,q}}
\le C\|A_1-A_2\|_{W^{1,\infty}}\|\rho\|_{L^\infty}.
$$
Finally,
$$
\mathcal T_{A_1}\rho-\mathcal T_{A_2}\rho
=-2(A_1-A_2)\partial_\nu\psi_2-2A_1\partial_\nu w.
$$
Taking the $C^\alpha(\mathbb T)$ norm and using the trace embedding proves
\eqref{eq:T-coefficient-Lipschitz}.  The special case
\eqref{eq:T-variable-perturbation} follows from
\eqref{eq:A-tau-lipschitz-close} with $A_1=A_{\tau,F}$ and $A_2=1$.
\end{proof}

The elliptic regularity applied above can be found in e.g. \cite[Chapter~9]{GT2001}.

\subsection{Estimates for singular perturbations}

We first record an auxiliary logarithmic-potential estimate, and then prove the two elementary perturbative estimates needed below. The auxiliary estimate bounds the Poisson balayage of $P[\sigma]$.
The first perturbative estimate gives a rate, after Poisson balayage, for the small singular mass. The second controls the imbalance direction in a weak
H\"older norm.

\begin{lem}\label{lem:Astar-log-potential}
Let $\sigma$ be a finite positive measure of class $A^*$ on $\mathbb T$.  Then
$$
\sup_{\zeta\in\mathbb T}\ct(P[\sigma])(\zeta)\le C(\sz)<\infty.
$$
\end{lem}

\begin{proof}
Recall that
$$
P_z(\zeta)=\frac{1-|z|^2}{|\zeta-z|^2}, \quad z\in \mathbb D, \ \zeta\in \mathbb T,
$$
and write
$$
 K(\zeta,\eta):=\int_{\mathbb D}P_z(\zeta)P_z(\eta)\,da(z).
$$
Then
\begin{equation}\label{eq:T-K}
\mathcal T(P[\sigma])(\zeta)=\int_{\mathbb T}K(\zeta,\eta)\,d\sigma(\eta).  
\end{equation}

We first estimate $K$.  Write $z=re^{i\theta}$, $\zeta=e^{i\varphi}$,
and $\eta=e^{is}$.  Since
$$
P_z(\zeta)=\sum_{k\in\mathbb Z}r^{|k|}e^{ik(\theta-\varphi)},
$$
the orthogonality of the Fourier modes and the normalization $da=2r\,dr\,dm(e^{i\theta})$
give
\begin{align*}
    K(\zeta,\eta)=& \ \sum_{k,\,l\in \mathbb Z} r^{|k|+|l|} e^{-ik\varphi}e^{-i l s} \int_{\mathbb T} e^{i(k+l)\theta}\,dm(e^{i\theta})\\
=& \  \sum_{k\in\mathbb Z}
\left(\int_0^1 2r^{2|k|+1}\,dr\right)e^{ik(s-\varphi)}
=  \sum_{k\in\mathbb Z}\frac{e^{ik(s-\varphi)}}{|k|+1}.
\end{align*}
Consequently, $K$ is the convolution kernel whose Fourier coefficients are 
$$\widehat K(k)=\frac{1}{|k|+1},$$
i.e. it is the logarithmic kernel on the circle with
\begin{equation}\label{eq:estimate-K}
0\le K(\zeta,\eta)
 \le C\left(1+\log\frac{2}{|\zeta-\eta|}\right).    
\end{equation}

Since $\sigma$ is a positive $A^\ast$-measure, its distribution function belongs to the Zygmund class. Hence by \eqref{eq:Astar-continuity}, its modulus of continuity satisfies
$$
\sigma(I(\zeta,h)) \le  C_\sigma h\log\frac{2}{h},  \qquad 0<h<1,
$$
for every arc $I(\zeta,h)$ centered at $\zeta$ of length   $2h$.
Fix $\zeta\in\mathbb T$, and set
$$
 A_j(\zeta) := \{\eta\in\mathbb T:2^{-j-1}<|\zeta-\eta|\le 2^{-j}\},  \qquad j=0,1,2,\ldots .
$$
On each $A_j(\zeta)$ one has
$$
 \log\frac{2}{|\zeta-\eta|}\le C(j+1),
$$
while the $A^\ast$-modulus estimate gives
$$
 \sigma(A_j(\zeta))  \le  C_\sigma 2^{-j}(j+1).
$$
Therefore, plugging this together with \eqref{eq:estimate-K} into \eqref{eq:T-K} yields
\begin{align*}
\mathcal  T(P[\sigma])(\zeta)
&\le C\sigma(\mathbb T)
+  C\sum_{j=0}^\infty (j+1)\sigma(A_j(\zeta)) \\
&\le C\sigma(\mathbb T)
+ C_\sigma\sum_{j=0}^\infty (j+1)^2 2^{-j}
<\infty .
\end{align*}
The bound is independent of $\zeta$ and thus we conclude the lemma.
\end{proof}

\begin{lem}\label{lem:singular-small-rate}
Let $\sigma$ be a finite positive $A^*$ measure on $\mathbb T$.  Let
$$
    E_a(z):=1-e^{-2aP[\sigma](z)}.
$$
Let $0<\alpha<\beta<1$. Then there exists $C=C(\az,\bz, \sz)$ so that
\begin{equation}\label{eq:singular-small-rate}
    \|\ct(E_a Q)\|_{C^\alpha(\mathbb T)}
    \le C a^{1-\alpha/\beta}\|Q\|_{L^\infty(\mathbb D)}
\end{equation}
for every $Q\in L^\infty(\mathbb D)$.
\end{lem}

\begin{proof}
Since $0\le E_a\le 2aP[\sigma]$, positivity of the Poisson kernel gives
$$
|\mathcal T(E_aQ)(\zeta)|
\le\|Q\|_{L^\infty(\mathbb D)}\,\ct (E_a)(\zeta)
\le    2a\|Q\|_{L^\infty(\mathbb D)}\,\ct (P[\sigma])(\zeta).
$$
Thus, Lemma~\ref{lem:Astar-log-potential} gives
$$
    \|\ct(E_aQ)\|_{L^\infty(\mathbb T)}
    \le C_\sz a\|Q\|_{L^\infty}.
$$
On the other hand,  since  $0\le E_a\le1$,  Lemma~\ref{lem:T-holder} gives
$$
    \|\ct(E_aQ)\|_{C^\beta(\mathbb T)}
    \le C_\beta\|Q\|_{L^\infty}.
$$
Using the interpolation inequality
$$
\|g\|_{C^\alpha(\mathbb T)}\le C_{\az,\bz}\|g\|_{L^\infty(\mathbb T)}^{1-\alpha/\beta}\|g\|_{C^\beta(\mathbb T)}^{\alpha/\beta},     \qquad 0<\alpha<\beta<1,
$$
we obtain \eqref{eq:singular-small-rate}.
\end{proof}

\begin{lem}\label{lem:imbalance-weak-bound}
Let $0<\gamma<1$ and $0<\alpha<\gamma$.  Suppose $\rho$ satisfies
$$|\rho(z)|\le |P[\nu](z)|e^{-2aP[\Sigma](z)},$$
where, recall $\nu$ and $\Sigma$ are defined in \eqref{eq:defn-nu-sigma} with $|P[\nu]|\le P[\Sigma]$.  Then
\begin{equation}\label{eq:imbalance-weak-bound}
\|\ct(\rho Q)\|_{C^\alpha(\mathbb T)} \le C a^{-\gamma}\|Q\|_{L^\infty(\mathbb D)}.
\end{equation}
\end{lem}

\begin{proof}
Set $q:=1-\gamma.$
Then $0<q<1$, and the assumption $\alpha<\gamma$ is the same as
$$
q+\alpha<1.
$$
Since $|P[\nu]|\le P[\Sigma]$, we have
$$
|\rho(z)| \le P[\Sigma](z)e^{-2aP[\Sigma](z)}.
$$

Note that, for every $X\ge0$,
$$
Xe^{-2aX} =a^{-\gamma}X^{1-\gamma}(aX)^\gamma e^{-2aX}
\le C_\gamma a^{-\gamma}X^{1-\gamma}.
$$
Therefore
\begin{equation}\label{eq:rho}
    |\rho(z)|\le C_\gamma a^{-\gamma}(P[\Sigma](z))^q,
\end{equation}
and it is enough to prove
\begin{equation}\label{eq: PQ}
\|\mathcal T((P[\Sigma])^q Q)\|_{C^\alpha(\mathbb T)}
\le C_\gz\|Q\|_{L^\infty(\mathbb D)} .
\end{equation}

Write $\mathcal B(z):=(P[\Sigma](z))^q.$
We first record an $L^{1/q}$-bound for $\mathcal  B$  on circles. Since the Poisson kernel preserves mass,
$$    
\int_{\mathbb T}P[\Sigma](re^{i\theta})\,dm(e^{i\theta})
=\Sigma(\mathbb T).
$$
Then
\begin{equation}\label{eq:mathcalB}
    \|\mathcal B(re^{i\theta})\|_{L^{1/q}(\mathbb T)}
    = \left( \int_{\mathbb T}P[\Sigma](re^{i\theta})\,dm(e^{i\theta})
    \right)^q
    = \Sigma(\mathbb T)^q.
\end{equation}

Let
$$
s:=\frac1{1-q}=\frac1\gamma ,
$$
so that $s'=\frac1 q$. Recall that the Poisson kernel satisfies
\begin{equation}\label{eq:Prs}
\|P_r\|_{L^s(\mathbb T)}
\le C(1-r)^{-q} \quad 
\text{ and }\quad 
\|\partial_tP_r\|_{L^s(\mathbb T)}
\le C(1-r)^{-q-1}.    
\end{equation}
Hence, for every $h$ with $|h|$ small
\begin{align}
\|P_r(\cdot+h)-P_r\|_{L^s(\mathbb T)}
&\le C\min\{(1-r)^{-q},\, |h|(1-r)^{-q-1}\} \notag\\
&\le C |h|^\alpha (1-r)^{-q-\alpha},\label{eq:Pr}
\end{align}
where in the last step we applied, for $a,b>0$, 
$$\min\{a,b\}\le a^{1-\alpha}b^\alpha.$$

Now write $g=\mathcal T(\mathcal B Q)$. Then for $\zeta=e^{i\varphi}\in \mathbb T$,
$$
 g(\zeta)=\int_0^1\int_{\mathbb T}
P_r(\varphi-\theta)\mathcal B(re^{i\theta})Q(re^{i\theta})
\,dm(e^{i\theta})\,2r\,dr .
$$
Therefore, applying \eqref{eq:mathcalB} and \eqref{eq:Pr} , via H\"older's inequality
\begin{align*}
|g(e^{i(\varphi+h)})-g(e^{i\varphi})|
&\le \|Q\|_{L^\infty} \int_0^1 \|P_r(\cdot+h)-P_r\|_{L^s(\mathbb T)}  \|\mathcal B(re^{i\theta})\|_{L^{s'}(\mathbb T)}\,2r\,dr \\
&\le  C_\gz\|Q\|_{L^\infty}|h|^\alpha \int_0^1(1-r)^{-q-\alpha}\,dr.
\end{align*}
The last integral is finite since $q+\alpha<1$. Thus
$$
[g]_{C^\alpha(\mathbb T)}
\le C_\gz\|Q\|_{L^\infty(\mathbb D)}.
$$
The $L^\infty$-bound follows in a similar way from \eqref{eq:Prs}
together with the integrability of $(1-r)^{-q}$ since $q<1$. Hence \eqref{eq: PQ} follows. Then multiplying \eqref{eq: PQ} by $C_\gamma a^{-\gamma}$ and applying \eqref{eq:rho}, we conclude \eqref{eq:imbalance-weak-bound}. 
\end{proof}

\subsection{The weighted DSS equation}
Let $f\in C^2((-\rho,\rho))$ and set
$f_0:=f(0)>0.$

\medskip
\paragraph{\it The weighted residual.}
For fixed small parameters $(W,a,\delta,\tau,\lambda)$, with
$W\in\mathcal Y_\alpha$, let $F=F_{W,a,\delta}$ be the conformal map defined in
\eqref{eq:hd-F-prime-corrected}.
\begin{lem}\label{lem:weighted-semilinear-disk-estimate}
Let $s=\lambda^2$.  For all sufficiently small
$s$, $\tau$, $a$, $|\delta|$, and $\|W\|_{C^\alpha}$, the boundary value problem
\begin{equation}\label{eq:hd-disk-semilinear-corrected}
    -\operatorname{div}(A_{\tau,F}\nabla v)
    =A_{\tau,F}|F'|^2 f(sv)
    \quad\text{in }\mathbb D,
    \qquad
    v=0\quad\text{on }\mathbb T
\end{equation}
has a unique bounded solution, denoted by $v_{W,a,\delta,\tau,\lambda}$.
Moreover, for the fixed $q>2$ chosen above,
$$
    \|v_{W,a,\delta,\tau,\lambda}\|_{L^\infty(\mathbb D)}
    +\|v_{W,a,\delta,\tau,\lambda}\|_{W^{2,q}(\mathbb D)}
    \le C,
$$
and
\begin{equation}\label{eq:hd-f-small-corrected}
    f(\lambda^2v_{W,a,\delta,\tau,\lambda})=f_0+O(\lambda^2)
    \quad\text{in }L^\infty(\mathbb D),
\end{equation}
uniformly for the remaining parameters in the small region.  In addition,
$W\mapsto v_{W,a,\delta,\tau,\lambda}$ is Frechet differentiable and
\begin{equation}\label{eq:D-W-v-bound}
    \|D_Wv_{W,a,\delta,\tau,\lambda}[H]\|_{L^\infty(\mathbb D)}
    \le C\|H\|_{C^\alpha(\mathbb T)}.
\end{equation}
Consequently,
\begin{equation}\label{eq:D-W-f-lambda2-bound}
    \|D_W(f(\lambda^2v_{W,a,\delta,\tau,\lambda}))[H]\|_{L^\infty(\mathbb D)}
    \le C\lambda^2\|H\|_{C^\alpha(\mathbb T)}.
\end{equation}
\end{lem}

\begin{proof}
Let $\mathcal G_A$ denote the zero-Dirichlet inverse of
$-\operatorname{div}(A\nabla\cdot)$ on $\mathbb D$.  By the maximum principle and
the preceding $W^{2,q}$ estimates,
$$
    \|\mathcal G_A g\|_{L^\infty}
    +\|\mathcal G_A g\|_{W^{2,q}}
    \le C\|g\|_{L^\infty}
$$
uniformly for $A=A_{\tau,F}$ in the small parameter region.  Put
$$
    B:=A_{\tau,F}|F'|^2 .
$$
By \eqref{eq:Fprime-upper-bound} and \eqref{eq:A-tau-lipschitz-close},
$0\le B\le C$.  Consider
$$
    \mathcal M(v):=\mathcal G_{A_{\tau,F}}(Bf(sv)).
$$
For $R>0$ large and $s$ small, $\mathcal M$ maps the ball
$\|v\|_{L^\infty}\le R$ into itself.  Moreover,
$$
\|\mathcal M(v_1)-\mathcal M(v_2)\|_{L^\infty}
\le Cs\|v_1-v_2\|_{L^\infty},
$$
so $\mathcal M$ is a contraction for $s$ sufficiently small.  This gives existence,
uniqueness, and the uniform $L^\infty$ bound.  The $W^{2,q}$ bound follows from
the equation.
Since $f\in C^2$ and $\|v\|_{L^\infty}\le C$,
$$
    \|f(sv)-f(0)\|_{L^\infty}\le Cs,
$$
which is \eqref{eq:hd-f-small-corrected}.

It remains to record the $W$-derivative bound.  Let
$\dot v=D_Wv[H]$, $\dot A=D_WA_{\tau,F}[H]$, and
$\dot B=D_W(A_{\tau,F}|F'|^2)[H]$.  Differentiating the equation gives
$$
-\operatorname{div}(A_{\tau,F}\nabla\dot v)
=\operatorname{div}(\dot A\nabla v)+\dot B f(sv)+Bsf'(sv)\dot v .
$$
The divergence term may be rewritten as
$$
\operatorname{div}(\dot A\nabla v)=\nabla\dot A\cdot\nabla v+
\dot A\Delta v .
$$
By \eqref{eq:A-tau-W-derivative}, Lemma~\ref{lem:hd-herglotz-stability}, and the
uniform $W^{2,q}$ bound for $v$, the right-hand side is bounded in $L^q$ by
$$
    C\|H\|_{C^\alpha}+Cs\|\dot v\|_{L^\infty}.
$$
The $W^{2,q}$ estimate and the embedding $W^{2,q}\hookrightarrow L^\infty$ give
$$
\|\dot v\|_{L^\infty}
\le C\|H\|_{C^\alpha}+Cs\|\dot v\|_{L^\infty}.
$$
For $s$ small we absorb the last term and obtain \eqref{eq:D-W-v-bound}.  Finally,
$$
    D_W(f(sv))[H]=sf'(sv)\dot v,
$$
which proves \eqref{eq:D-W-f-lambda2-bound}.
\end{proof}

For the chosen small parameters, let $v_{W,a,\delta,\tau,\lambda}$ denote the
solution given by Lemma~\ref{lem:weighted-semilinear-disk-estimate}. 
Recall the definition of $F'$ in \eqref{eq:hd-F-prime-corrected}.
Define the interior density
$$
    \rho_{W,a,\delta,\tau,\lambda}
    :=A_{\tau,F}|F'|^2 f(\lambda^2v_{W,a,\delta,\tau,\lambda}).
$$
The weighted boundary equation is
$$
    \mathcal T_{\tau,F}(\rho_{W,a,\delta,\tau,\lambda})
    =C\,(A_{\tau,F})^*|(F'{})^*|
    \quad\text{on }\mathbb T
$$
for some positive constant $C$, where $^*$ denotes the boundary trace on $\mathbb T$.  Equivalently, after subtracting the boundary mean,
we need to solve
\begin{equation}\label{eq:hd-residual-corrected}
\Phi(W,a,\delta,\tau,\lambda)
:=\mathcal P_0\left[ \log\mathcal T_{\tau,F}(\rho_{W,a,\delta,\tau,\lambda})
-\log\bigl((A_{\tau,F})^*|(F'{})^*|\bigr)
\right]=0,
\end{equation}
where we recall $\mathcal P_0$  denotes the mean-zero projection; one can compare with the planar case \eqref{eq:Psi-def}.

\begin{lem}\label{lem:uniform-log-lower-bound}
After shrinking the small parameter region, there exists $c_0>0$ such that
\begin{equation}\label{eq:uniform-log-lower-bound}
    \mathcal T_{\tau,F}(\rho_{W,a,\delta,\tau,\lambda})\ge c_0
    \qquad\text{on }\mathbb T
\end{equation}
whenever $\|W\|_{C^\alpha}$, $a$, $\tau$, and $\lambda$ are sufficiently small and
$|\delta|\le a/2$.
\end{lem}

\begin{proof}
On $|z|\le1/2$, the Poisson kernel is bounded above uniformly in the boundary
point.  Since $\mu_{a,\delta}(\mathbb T)\le C a$ for $|\delta|\le a/2$, we have
$$
    P[\mu_{a,\delta}](z)\le Ca
    \qquad\text{for } |z|\le1/2.
$$
Also $|P[W](z)|\le\|W\|_{L^\infty}$.  Hence
$$
    |F'(z)|^2
    =e^{-2P[W](z)-2P[\mu_{a,\delta}](z)}
    \ge c>0
    \qquad\text{for } |z|\le1/2.
$$
Moreover $A_{\tau,F}\ge1/2$ for $\tau$ small, and
Lemma~\ref{lem:weighted-semilinear-disk-estimate} gives
$f(\lambda^2v_{W,a,\delta,\tau,\lambda})\ge f_0/2$ for $\lambda$ small.  Thus
there exists $m_0>0$ such that
$$
    \rho_{W,a,\delta,\tau,\lambda}(z)\ge m_0
    \qquad\text{for } |z|\le1/2.
$$
For the unweighted operator,
$$
\ct\rho(\zeta)=\int_{\mathbb D}P_z(\zeta)\rho(z)\,da(z).
$$
If $|z|\le1/2$, then
$$
    P_z(\zeta)=\frac{1-|z|^2}{|\zeta-z|^2}\ge\frac13,
    \qquad \zeta\in\mathbb T.
$$
Therefore
$$
    \ct\rho(\zeta)
    \ge m_0\int_{|z|\le1/2}P_z(\zeta)\,da(z)
    \ge c_1>0.
$$
On the other hand, \eqref{eq:T-variable-perturbation} gives
$$
\|\mathcal T_{\tau,F}\rho-\ct\rho\|_{C^\alpha(\mathbb T)}
\le C\tau\|\rho\|_{L^\infty(\mathbb D)}.
$$
Taking $\tau$ sufficiently small gives
$\mathcal T_{\tau,F}\rho\ge c_1/2$.  This proves the lemma.
\end{proof}

Thus the logarithms in \eqref{eq:hd-residual-corrected} are well-defined and the map $X\mapsto\log X$ is uniformly Lipschitz, with uniformly bounded first and second derivatives, on the range of the positive functions appearing below.
 

\medskip
\paragraph{\it Linearization at the disk.}

We next compute the linearization in the outer variable $W$ at the disk.  This is the source for the high-mode invertibility.  Put
$$
    W=0,\qquad a=\delta=\tau=\lambda=0.
$$
Then $F(z)=z$, $A_{\tau,F}\equiv1$. Moreover, for a small outer perturbation $W$, since $\mathcal P_0$ subtracts the constant $\log(f_0)$, the residual is
$$
\Phi(W)=\mathcal P_0\left[\log \ct\bigl(f_0e^{-2P[W]}\bigr)-\log \bigl( e^{- W}\bigr)\right]=\mathcal P_0\left[\log \ct\bigl(e^{-2P[W]}\bigr)+W\right].
$$
In addition, we decompose this residual into its high-mode part and its sine-mode coefficient:
$$
\Phi_{\rm hi}:=\mathcal P_{\rm hi}\Phi,
 \qquad \Phi_1:=\pi_{\sin}\Phi .
$$
Thus $\Phi_{\rm hi}$ takes values in the high-mode subspace
$\mathcal P_{\rm hi}C^\alpha(\mathbb T)$, while $\Phi_1$ is a scalar.

Let $H\in C^\alpha(\mathbb T)$ have mean zero.  Since
$$
    e^{-2P[sH]}=1-2sP[H]+O(s^2)
$$
and $\ct(1)=1$, we get
$$
    \frac{d}{ds}\bigg|_{s=0}
    \log \ct\bigl(e^{-2P[sH]}\bigr)
    =-2\ct(P[H]).
$$
Therefore
\begin{equation}\label{eq:disk-linearization-W}
    D_W\Phi(0)[H]=H-2\ct(P[H]).
\end{equation}
We write
$  K:=\ct\circ P.$
Thus the disk linearization is
$$
    D_W\Phi(0)=I-2K.
$$

We now diagonalize $K$ on Fourier modes.  Let
$$
    \mathbf e_k(e^{it}):=e^{ikt},\qquad k\in\mathbb Z\setminus\{0\}.
$$
The Poisson extension of $e_k$ is
$$
    P[\mathbf e_k](re^{it})=r^{|k|}e^{ikt}.
$$
Using the definition of $\ct$ and the normalization $da=2r\,dr\,dm(e^{it})$, we find
$$
    K(\mathbf e_k)
    =\ct(P[\mathbf e_k])
    =\left(\int_0^1 2r^{2|k|+1}\,dr\right)\mathbf e_k
    =\frac1{|k|+1}\mathbf e_k .
$$
Consequently
\begin{equation}\label{eq:disk-multiplier}
    (I-2K)\mathbf e_k
    =\left(1-\frac2{|k|+1}\right)\mathbf e_k
    =\frac{|k|-1}{|k|+1}\mathbf e_k .
\end{equation}
Thus, the only zero eigenvalues occur for $|k|=1$.  In the reflected space
$\mathcal X_\alpha$, the $\cos t$ mode is forbidden by the symmetry
$W(-\overline\zeta)=W(\zeta)$, while $\sin t$ is allowed.  By definition of
$\mathcal Y_\alpha$, the $\sin t$ coefficient is also removed.  Therefore no first
Fourier mode remains in $\mathcal Y_\alpha$.

\medskip
\paragraph{\it High-mode invertibility.}
We next show that the operator $K$ is compact on the symmetry class in $C^\alpha(\mathbb T)$: Indeed,
$P$ maps $C^\alpha(\mathbb T)$ boundedly into bounded harmonic functions in the disk,
$\ct$ maps bounded functions into $C^\beta(\mathbb T)$ for every $\beta<1$ by
Lemma~\ref{lem:T-holder}, and the embedding $C^\beta\hookrightarrow C^\alpha$ is
compact whenever $\beta>\alpha$.  Hence the Fredholm alternative and the absence of
kernel on $\mathcal Y_\alpha$ imply that
\begin{equation}\label{eq:high-mode-invertibility}
    I-2K:\mathcal Y_\alpha\to \mathcal P_{\rm hi}C^\alpha(\mathbb T)
\end{equation}
is an isomorphism.  This implies the high-mode invertibility.

\medskip

\paragraph{\it Solving the high modes.}
Fix $\beta\in(\alpha,1)$ and put $\theta:=1-\frac{\alpha}{\beta}. $
We choose $\beta$ sufficiently close to $1$, and then choose $\gamma$ so that
\begin{equation}\label{eq:alpha-beta-gamma-condition}
    \alpha<\gamma<\theta.
\end{equation}
This is possible by \eqref{eq:alpha-choice-hd}.  In the perturbative estimates below
we keep
\begin{equation}\label{eq:delta-half-range}
    |\delta|\le\frac a2 .
\end{equation}

\begin{lem}\label{lem:hd-log-residual-perturbation}
For all sufficiently small parameters satisfying \eqref{eq:delta-half-range},
\begin{equation}\label{eq:high-mode-center-residual}
    \|\Phi_{\rm hi}(0,a,\delta,\tau,\lambda)\|_{C^\alpha(\mathbb T)}
    \le C\left(a^\theta+a^{-\gamma}|\delta|+\tau+\lambda^2\right).
\end{equation}
Moreover, uniformly for $W\in\mathcal Y_\alpha$ with
$\|W\|_{C^\alpha}\le1$,
\begin{equation}\label{eq:high-mode-parameter-increment}
\|\Phi_{\rm hi}(W,a,\delta,\tau,\lambda)
       -\Phi_{\rm hi}(W,a,0,0,0)\|_{C^\alpha(\mathbb T)}   
    \le C\left(a^{-\gamma}|\delta|+\tau+\lambda^2\right).
\end{equation}
Finally, for $H\in\mathcal Y_\alpha$ define the full mean-zero derivative error
$$
    \mathcal R_{W,a,\delta,\tau,\lambda}H
    :=D_W\Phi(W,a,\delta,\tau,\lambda)[H]-(I-2K)H .
$$
Then
\begin{equation}\label{eq:full-W-derivative-perturbation}
    \|\mathcal R_{W,a,\delta,\tau,\lambda}\|_{\mathcal L(\mathcal Y_\alpha,
       C^\alpha(\mathbb T))}
    \le C\left(\|W\|_{C^\alpha}+a^\theta+\tau+\lambda^2\right).
\end{equation}
Consequently,
\begin{equation}\label{eq:high-mode-derivative-perturbation}
\begin{split}
\left\|D_W\Phi_{\rm hi}(W,a,\delta,\tau,\lambda)
-(I-2K)\right\|_{\mathcal L(\mathcal Y_\alpha,
\mathcal P_{\rm hi}C^\alpha)}    
\le C\left(\|W\|_{C^\alpha}+a^\theta+\tau+\lambda^2\right).
\end{split}
\end{equation}
\end{lem}

\begin{proof}
Throughout the proof, for simplification, write
$$
    \mu=\mu_{a,\delta},\qquad  A=A_{\tau,F},\qquad
    G=e^{-2P[W]-2P[\mu]},\qquad M=f(\lambda^2v_{W,a,\delta,\tau,\lambda}),
$$
and set
$$
    R:=AGM,  \qquad  B:=\mathcal T_{\tau,F}R.
$$
Then
\begin{equation}\label{eq:Phi-through-B}
    \Phi(W,a,\delta,\tau,\lambda)
    =\mathcal P_0\left[\log B-\log A^*+W\right],
\end{equation}
because the singular inner factor has boundary modulus one a.e. and
$|(F')^*|=e^{-W}$ on $\mathbb T$.
By Lemma~\ref{lem:uniform-log-lower-bound}, $B\ge c_0>0$ in the whole small
parameter region.  Thus the logarithm is uniformly Lipschitz and has uniformly
bounded first derivative on the range considered here.

We first record a full mean-zero perturbation estimate which is stronger than the
one needed for the high modes.  Since $|\delta|\le a/2$,
$$
    0\le \mu_{a,\delta}\le {\frac{3a} 2}\Sigma .
$$
Hence
$$
    0\le 1-e^{-2P[\mu_{a,\delta}]}
    \le 1-e^{-3aP[\Sigma]}.
$$
Lemma~\ref{lem:singular-small-rate}, applied with $3a/2$ in place of $a$, gives
for every bounded $Q$,
\begin{equation}\label{eq:mu-small-balanced-use}
    \left\|\ct\left((1-e^{-2P[\mu_{a,\delta}]})Q\right)\right\|_{C^\alpha}
    \le Ca^\theta\|Q\|_{L^\infty(\mathbb D)} .
\end{equation}
Together with \eqref{eq:T-variable-perturbation},
\eqref{eq:hd-f-small-corrected}, and the pointwise bound
$|e^{-2P[W]}-1|\le C\|W\|_{C^\alpha}$ for $\|W\|_{C^\alpha}\le1$, this yields
\begin{equation}\label{eq:B-minus-f0}
    \|B-f_0\|_{C^\alpha(\mathbb T)}
    \le C\left(\|W\|_{C^\alpha}+a^\theta+\tau+\lambda^2\right).
\end{equation}
Indeed, write
$$
 B-f_0=(\mathcal T_{\tau,F}-\ct)R+
        \ct(R-f_0),
$$
and decompose $R-f_0$ into the four contributions coming from
$A-1$, $e^{-2P[W]}-1$, $e^{-2P[\mu]}-1$, and
$f(\lambda^2v)-f_0$.

The same argument, but with the fundamental theorem of calculus in the imbalance
parameter, gives the sharper increment estimate in the $\delta$-direction.  More
precisely,
$$
 e^{-2P[\mu_{a,\delta}]}-e^{-2aP[\Sigma]}
 =\int_0^\delta -2P[\nu]e^{-2P[a\Sigma+s\nu]}\,ds .
$$
If $|s|\le |\delta|\le a/2$, then
$a\Sigma+s\nu\ge (a/2)\Sigma$ and $|P[\nu]|\le P[\Sigma]$.  Hence
Lemma~\ref{lem:imbalance-weak-bound}, applied with $a/2$ in place of $a$, gives
\begin{equation}\label{eq:delta-increment-bound}
    \left\|\ct\left(
    \bigl(e^{-2P[\mu_{a,\delta}]}-e^{-2aP[\Sigma]}\bigr)Q\right)
    \right\|_{C^\alpha}
    \le Ca^{-\gamma}|\delta|\|Q\|_{L^\infty(\mathbb D)} .
\end{equation}
Combining this with \eqref{eq:T-variable-perturbation} and
\eqref{eq:hd-f-small-corrected}, and then using the uniform lower bound for $B$,
proves \eqref{eq:high-mode-parameter-increment}.  Taking $W=0$ in the same
estimate and adding the balanced singular estimate
\eqref{eq:mu-small-balanced-use} proves \eqref{eq:high-mode-center-residual}.

\medskip

It remains to prove the derivative estimate.  We apply the following direct consequence of Lemma~\ref{lem:weighted-trace-stability} and \eqref{eq:A-tau-W-derivative}: For every bounded density $\rho$,
\begin{equation}\label{eq:trace-W-derivative-bound}
    \|D_W(\mathcal T_{\tau,F})[\rho][H]\|_{C^\alpha(\mathbb T)}\le C\tau\|\rho\|_{L^\infty(\mathbb D)}\|H\|_{C^\alpha(\mathbb T)} .
\end{equation}
For completeness we recall the proof here:  If $\psi$ solves
$$-\operatorname{div}(A\nabla\psi)=\rho \ \text{ and } \ \dot A=D_WA_{\tau,F}[H],$$
then the coefficient variation $\dot\psi$ solves
$$
    -\operatorname{div}(A\nabla\dot\psi)
    =\operatorname{div}(\dot A\nabla\psi),
    \qquad \dot\psi|_{\mathbb T}=0.
$$
The $W^{2,q}$ estimate used in Lemma~\ref{lem:weighted-trace-stability}, together
with \eqref{eq:A-tau-W-derivative}, gives
$\|\dot\psi\|_{W^{2,q}}\le C\tau\|\rho\|_{L^\infty}\|H\|_{C^\alpha}$, and hence
\eqref{eq:trace-W-derivative-bound} follows from
$D_W(\mathcal T_{\tau,F})[\rho][H]
=-2\dot A\partial_\nu\psi|_{\mathbb T}-2A\partial_\nu\dot\psi|_{\mathbb T}$.

We next compare $D_WB[H]$ with the disk derivative $-2f_0KH$.  Since
$$
    D_WR[H] =-2P[H]R+(D_WA[H])GM +AG\lambda^2 f'(\lambda^2v)D_Wv[H],
$$
Lemma~\ref{lem:weighted-semilinear-disk-estimate},
\eqref{eq:A-tau-W-derivative}, and \eqref{eq:trace-W-derivative-bound} give
\begin{equation}\label{eq:DB-main-comparison}
    \|D_WB[H]+2f_0KH\|_{C^\alpha(\mathbb T)} \le C\left(\|W\|_{C^\alpha}+a^\theta+\tau+\lambda^2\right)
       \|H\|_{C^\alpha(\mathbb T)} .
\end{equation}
Indeed, the main term is
$\ct(-2f_0P[H])=-2f_0KH$.  The error from replacing
$\mathcal T_{\tau,F}$ by $\ct$ is $O(\tau\|H\|_{C^\alpha})$.  The error from
$R-f_0$ is estimated by the same decomposition as in \eqref{eq:B-minus-f0}, with
$P[H]$ absorbed into the bounded factor $Q$.  The remaining two differentiated
terms are bounded respectively by $C\tau\|H\|_{C^\alpha}$ and
$C\lambda^2\|H\|_{C^\alpha}$.

Using \eqref{eq:B-minus-f0}, \eqref{eq:DB-main-comparison}, and the lower bound
$B\ge c_0$, we obtain
\begin{equation}\label{eq:DlogB-comparison}
    \left\|{\frac{D_WB[H]}  B}+2KH\right\|_{C^\alpha(\mathbb T)} \le \left(\|W\|_{C^\alpha}+a^\theta+\tau+\lambda^2\right)
       \|H\|_{C^\alpha(\mathbb T)} .
\end{equation}
Finally,
$$
    \|D_W\log A^*[H]\|_{C^\alpha(\mathbb T)}  \le C\tau\|H\|_{C^\alpha(\mathbb T)}
$$
by \eqref{eq:A-tau-W-derivative}.  Differentiating
\eqref{eq:Phi-through-B} therefore gives
$$
D_W\Phi(W,a,\delta,\tau,\lambda)[H]-(I-2K)H
 =\mathcal P_0\left[ \frac{D_WB[H]} {B}+2KH-D_W\log A^*[H]
   \right].
$$
Together with \eqref{eq:DlogB-comparison}, this proves
\eqref{eq:full-W-derivative-perturbation}.  Since $(I-2K)H$ has no mean and no first sine mode when $H\in\mathcal Y_\alpha$, applying $\mathcal P_{\rm hi}$ gives \eqref{eq:high-mode-derivative-perturbation}.
\end{proof}

Let
$$
    \mathfrak L:=(I-2K)|_{\mathcal Y_\alpha}.
$$
By \eqref{eq:high-mode-invertibility}, $\mathfrak L$ is an isomorphism from
$\mathcal Y_\alpha$ onto $\mathcal P_{\rm hi}C^\alpha(\mathbb T)$.  The contraction
mapping theorem applied to
$$
W\mapsto -\mathfrak L^{-1}\left(\Phi_{\rm hi}(0,a,\delta,\tau,\lambda)
+\Phi_{\rm hi}(W,a,\delta,\tau,\lambda)-\Phi_{\rm hi}(0,a,\delta,\tau,\lambda)
-\mathfrak L W\right)
$$
and Lemma~\ref{lem:hd-log-residual-perturbation} imply that, when
$$
    a^\theta+a^{1-\gamma}+\tau+\lambda^2
$$
is sufficiently small, the high-mode equation
$$
    \Phi_{\rm hi}(W,a,\delta,\tau,\lambda)=0
$$
has a unique solution
$W=W(a,\delta,\tau,\lambda)\in\mathcal Y_\alpha$ in a fixed small ball around the
origin.  Moreover,
\begin{equation}\label{eq:W-highmode-size}
    \|W(a,\delta,\tau,\lambda)\|_{C^\alpha}
    \le C\left(a^\theta+a^{-\gamma}|\delta|+\tau+\lambda^2\right).
\end{equation}
If $W_a:=W(a,0,0,0)$, then
\begin{equation}\label{eq:W-minus-Wa-size}
    \|W(a,\delta,\tau,\lambda)-W_a\|_{C^\alpha}
    \le C\left(a^{-\gamma}|\delta|+\tau+\lambda^2\right).
\end{equation}

We need the following high-to-sine coupling estimate:
\begin{equation}\label{eq:high-to-sine-coupling}
    \|D_W\Phi_1(W_a,a,0,0,0)\|_{\mathcal L(\mathcal Y_\alpha,\mathbb R)}
    \le Ca^\theta .
\end{equation}
Indeed, at the disk one has
$$
\pi_{\sin}(I-2K)H=0, \qquad H\in\mathcal Y_\alpha,
$$
because $H$ has no first sine mode and $K$ is diagonal on Fourier modes.
Therefore
$$
D_W\Phi_1(W_a,a,0,0,0)[H]
=\pi_{\sin}\left( D_W\Phi(W_a,a,0,0,0)[H]-(I-2K)H\right).
$$
The full derivative estimate \eqref{eq:full-W-derivative-perturbation}, together with $\|W_a\|_{C^\alpha}\le Ca^\theta$, gives \eqref{eq:high-to-sine-coupling}.

\medskip
\paragraph{\it Reduction to the sine mode.}
We now reduce the remaining equation to the scalar sine mode.  Recall that
$$
    W_a:=W(a,0,0,0),
    \qquad \|W_a\|_{C^\alpha}\le Ca^\theta .
$$
For the balanced parameters $\delta=\tau=\lambda=0$, the exact antipodal symmetry
of $\Sigma$ implies that the balanced problem has no sine-mode residual.  Indeed, let
$$
    (\mathscr A W)(\zeta):=W(-\zeta),
    \qquad \zeta\in\mathbb T.
$$
Then $\mathscr A$ preserves $\mathcal Y_\alpha$.  Since $A_\#\Sigma=\Sigma$, the
balanced residual is equivariant:
$$
    \Phi(\mathscr A W,a,0,0,0)=\mathscr A\Phi(W,a,0,0,0).
$$
Consequently
$$
    \Phi_{\rm hi}(\mathscr A W,a,0,0,0)
    =\mathscr A\Phi_{\rm hi}(W,a,0,0,0).
$$
Thus, if $W_a$ solves the balanced high-mode equation, so does $\mathscr A W_a$.  As both of them lie in the small ball of the high-mode fixed point, the uniqueness of the solution yields 
$$
 \mathscr A W_a=W_a.
$$
Therefore $\Phi(W_a,a,0,0,0)$ is antipodally even.  Its mean-zero part is already
projected, its high-mode part vanishes by the definition of $W_a$, and its first
sine coefficient vanishes since  $\sin(t+\pi)=-\sin t$.  Hence
\begin{equation}\label{eq:balanced-full-residual-zero}
    \Phi(W_a,a,0,0,0)=0.
\end{equation}

Let
$$
A_a:=D_W\Phi_{\rm hi}(W_a,a,0,0,0)\colon \mathcal Y_\alpha\to \mathcal P_{\rm hi}C^\alpha(\mathbb T).
$$
By \eqref{eq:high-mode-derivative-perturbation}, $A_a$ is invertible for small $a$,
and $\|A_a^{-1}\|$ remains bounded as $a\to 0$.  
For $|\delta|\le a/2$, the map
$$
    W=W(a,\delta,\tau,\lambda)
$$
constructed above is the unique small solution of
$$
    \Phi_{\rm hi}(W(a,\delta,\tau,\lambda),a,\delta,\tau,\lambda)=0.
$$
Moreover, \eqref{eq:W-minus-Wa-size} gives the quantitative estimate
$$
    \|W(a,\delta,\tau,\lambda)-W_a\|_{C^\alpha}
    \le C\left(a^{-\gamma}|\delta|+\tau+\lambda^2\right).
$$
Define the scalar reduced residual in the variable $s=\lambda^2$ by
\begin{equation}\label{eq:scalar-reduced-residual}
\Psi(a,\delta,\tau,s):=\Phi_1(W(a,\delta,\tau,\sqrt{s}),a,\delta,\tau,\sqrt{s}),    \qquad s\ge0.
\end{equation}
The map is $C^1$ in $(\delta,\tau,s)$ in the small parameter region.  In the computations below all first derivatives are evaluated at
$(W,a,\delta,\tau,s)=(W_a,a,0,0,0)$.

Now let us compute the first derivatives of $\Psi$ at $(\delta,\tau,\lambda)=(0,0,0)$, with
$a>0$ fixed.  Differentiating the high-mode equation with respect to $\delta$ gives
$$
    D_\delta W(a,0,0,0)
    =-A_a^{-1}D_\delta\Phi_{\rm hi}(W_a,a,0,0,0).
$$
Therefore
\begin{equation}\label{eq:schur-delta}
    \partial_\delta\Psi(a,0,0,0)
    =D_\delta\Phi_1
    -D_W\Phi_1\,A_a^{-1}D_\delta\Phi_{\rm hi},
\end{equation}
where all derivatives on the right are evaluated at $(W,a,\delta,\tau,\lambda)=(W_a,a,0,0,0)$. Observe that, the first term on the right-hand side of \eqref{eq:schur-delta} is the direct effect of changing $\dz$ on the sine mode, and the second term is an indirect high-mode correction: Changing $\delta$
also creates high-mode residuals, and one needs the compensating change in $W$ to feed back into the scalar sine equation.

\medskip
\paragraph{\it The imbalance derivative.}
Recall that each $\sz_\pm$ is supported in an arc $I_\pm$ of length $2\eta$. In what follows, the notation $O(\eta)$ refers to errors that tend to zero
as the supporting arcs $I_\pm$ shrink to $\pm i$.  Once $\eta$ and the
measures $\sigma_\pm$ are fixed, all constants in the subsequent estimates may depend on this choice.

We claim that
\begin{equation}\label{eq:delta-schur-coeff}
    \partial_\delta\Psi(a,0,0,0)=4+O(\eta)+o_a(1),
\end{equation}
where $o_a(1)\to0$ as $a\to0$, after the arc-size $\eta$ has been fixed.
Put
$$
    G_a:=\exp\{-2P[W_a]-2aP[\Sigma]\}.
$$
At $(W,\delta,\tau,\lambda)=(W_a,0,0,0)$, the direct $\delta$-derivative of the logarithmic residual is
\begin{equation}\label{eq:direct-delta-derivative}
    D_\delta\Phi_1
    =
    \pi_{\sin}\left(
        \frac{\ct(-2P[\nu]G_a)}{\ct(G_a)}
      \right).
\end{equation}
Here the boundary denominator $|(F')^*|=e^{-W_a}$ has no $\delta$-derivative, since  the singular factor has boundary modulus one almost everywhere.

Since $\|W_a\|_{C^\alpha}=O(a^\theta)$, Lemma~\ref{lem:singular-small-rate} gives
$$
    \|\ct(G_a)-1\|_{C^\alpha}=o_a(1).
$$
Moreover, using $|P[\nu]|\le P[\Sigma]$, Fubini, and dominated convergence in the scalar pairing with $P[\sin t](z)=r\sin t$, we have
$$
    \pi_{\sin}\ct(P[\nu]G_a)
    =
    \pi_{\sin}\ct(P[\nu])+o_a(1).
$$
Therefore
\begin{equation}\label{eq:direct-delta-leading}
    D_\delta\Phi_1
    =
    -2\pi_{\sin}\ct(P[\nu])+o_a(1).
\end{equation}
Finally,
$$
\begin{aligned}
\pi_{\sin}\ct(P[\nu])
&=2\int_{\mathbb T}\ct(P[\nu])(e^{it})\sin t\,dm(e^{it})  \\
&=2\int_{\mathbb D}P[\nu](z)P[\sin t](z)\,da(z)  \\
&=2\int_{\mathbb T}\ct(P[\sin t])\,d\nu
 =\int_{\mathbb T}\sin t\,d\nu(t).
\end{aligned}
$$
Since $\sigma_+$ is a probability measure supported in an $\eta$-arc around $i$ and $\sigma_-=A_\#\sigma_+$ is supported around $-i$,
$$
    \int_{\mathbb T}\sin t\,d\nu(t)
    =
    \int_{\mathbb T}\sin t\,d\sigma_-(t)
      -
    \int_{\mathbb T}\sin t\,d\sigma_+(t)
    =
    -2+O(\eta).
$$
Thus the direct term in \eqref{eq:schur-delta} equals
$$
    D_\delta\Phi_1=4+O(\eta)+o_a(1).
$$

It remains to bound the indirect high-mode correction in \eqref{eq:schur-delta}.
By \eqref{eq:high-to-sine-coupling},
$$
    \|D_W\Phi_1(W_a,a,0,0,0)\|_{\mathcal L(\mathcal Y_\alpha,\mathbb R)} \le Ca^\theta .
$$
On the other hand, Lemma~\ref{lem:imbalance-weak-bound} gives
$$
    \|D_\delta\Phi_{\rm hi}(W_a,a,0,0,0)\|_{C^\alpha}
    \le Ca^{-\gamma}.
$$
Since $\|A_a^{-1}\|$ remains uniformly bounded and $\gamma<\theta$,
$$
 \left|D_W\Phi_1A_a^{-1}D_\delta\Phi_{\rm hi}\right|
\le C a^{\theta-\gamma}=o_a(1).
$$
This proves \eqref{eq:delta-schur-coeff}.

\medskip
\paragraph{\it The axial-weight derivative and the choice of $\delta$.}
Similarly,
\begin{equation}\label{eq:schur-tau}
    \partial_\tau\Psi(a,0,0,0)
    =
    D_\tau\Phi_1-D_W\Phi_1A_a^{-1}D_\tau\Phi_{\rm hi}.
\end{equation}
The second term in \eqref{eq:schur-tau} is $o_a(1)$, since $D_\tau\Phi_{\rm hi}$ is uniformly bounded and
$$
    \|D_W\Phi_1\|=O(a^\theta).
$$

It remains to compute the direct disk contribution.  Write $y=\operatorname{Im}z$.  At the disk,
$$
    A_{\tau,z}=(1+\tau y)^p=1+\tau py+O(\tau^2),
    \qquad
    (A_{\tau,z})^*=1+\tau p\sin t+O(\tau^2).
$$
Thus
\begin{equation}\label{eq:tau-direct-disk}
\left.\frac{d}{d\tau}\right|_{\tau=0}
\left[
\log\mathcal T_{\tau,z}(A_{\tau,z})
-
\log (A_{\tau,z})^*
\right]
=
\ct(py)+p\dot{\mathcal T}_y(1)-p\sin t,
\end{equation}
where $\dot{\mathcal T}_y$ denotes the first variation of the trace operator in the coefficient direction $y$.

We now compute $\dot{\mathcal T}_y(1)$.  Let $A_\varepsilon=1+\varepsilon y$, and let $\psi_\varepsilon$ solve
$$
    -\operatorname{div}(A_\varepsilon\nabla\psi_\varepsilon)=1
    \quad\hbox{in }\mathbb D,
    \qquad
    \psi_\varepsilon=0\quad\hbox{on }\mathbb T.
$$
Write
$$
\psi_\varepsilon=\psi_0+\varepsilon\psi_1+o(\varepsilon).
$$
Then
$\psi_0=\frac{1-r^2}{4},$
and the first variation satisfies
$$
    -\Delta\psi_1
    =  \operatorname{div}(y\nabla\psi_0)
    =  -\frac32 y,
    \qquad  \psi_1|_{\mathbb T}=0.
$$
Since $y=r\sin t$, we obtain
$$
    \psi_1=\frac{3}{16}(r^3-r)\sin t,
    \qquad  \partial_\nu\psi_1|_{\mathbb T}=\frac38\sin t.
$$
Therefore, using
$$
    \mathcal T_\varepsilon(1)
    = -2A_\varepsilon\partial_\nu\psi_\varepsilon|_{\mathbb T},
$$
we get
\begin{equation*}\label{eq:T-y-variation}
    \dot{\mathcal T}_y(1)= -2\left(y\partial_\nu\psi_0+\partial_\nu\psi_1\right)\big|_{\mathbb T}
    =-2\left(-\frac12\sin t+\frac38\sin t\right)
    =\frac14\sin t.
\end{equation*}
Also
$$
    \ct(y)=\ct(r\sin t)=\frac12\sin t.
$$
Hence the right-hand side of \eqref{eq:tau-direct-disk} is
$$
    \frac p2\sin t+\frac p4\sin t-p\sin t
    = -\frac p4\sin t.
$$
Taking the scalar $\sin t$-coefficient gives
\begin{equation}\label{eq:tau-schur-coeff}
    \partial_\tau\Psi(a,0,0,0)
    =  -\frac p4+o_a(1).
\end{equation}
Combining \eqref{eq:delta-schur-coeff} with \eqref{eq:tau-schur-coeff}, Taylor
expansion in $(\delta,\tau,s)$ gives
\begin{equation}\label{eq:reduced-sine-expansion}
    \Psi(a,\delta,\tau,s)
    =  \bigl(4+O(\eta)+o_a(1)\bigr)\delta
    -  \left(\frac p4+o_a(1)\right)\tau
    +  O_{a,\eta}(\delta^2+\tau^2+|\delta|\tau+s).
\end{equation}
Choose $\eta>0$ sufficiently small and then choose $a>0$ sufficiently small.  Then
$$
    \partial_\delta\Psi(a,0,0,0)\ge2.
$$
The implicit function theorem applied to the variables $(\delta,\tau,s)$ gives
$s_0(a)>0$, $\tau_0(a)>0$, and a unique small function
$ \delta=\delta_{\tau,s}$
such that
$$
    \Psi(a,\delta_{\tau,s},\tau,s)=0
$$
for $0\le\tau\le\tau_0(a)$ and $0\le s\le s_0(a)$.  Moreover,
\begin{equation}\label{eq:delta-solved-size}
    \delta_{\tau,s}
    =\frac{\frac p4+o_a(1)}{4+O(\eta)+o_a(1)}\,\tau
    +O_{a,\eta}(\tau^2+s).
\end{equation}
Equivalently,
$$
    \delta_{\tau,s}
    =\frac p{16}\tau+O(\eta\tau)+o_a(\tau)+O_{a,\eta}(\tau^2+s).
$$
Now choose $0<\tau\ll a$ and $0<s=\lambda^2\ll a$ so small that $|\delta_{\tau,s}|<\frac a 2.$

Set $\delta_{\tau,\lambda}:=\delta_{\tau,\lambda^2}$ for simplicity. 
Then $\mu_{a,\delta_{\tau,\lambda}}$ is a nonnegative nonzero singular measure,
since  $|\delta_{\tau,\lambda}|<a/2<a$.
Together with the high-mode equation, this gives 
\begin{equation}\label{eq:hd-weighted-boundary-equation-corrected}
    \mathcal T_{\tau,F}(\rho_{W,a,\delta_{\tau,\lambda},\tau,\lambda})
    =C_{\tau,\lambda}(A_{\tau,F})^*|(F'{})^*|
    \quad\text{on }\mathbb T
\end{equation}
for some $C_{\tau,\lambda}>0$.  
Since $ |\delta_{\tau,\lambda}|<a/2<a $, the singular measure
$\mu_{a,\delta_{\tau,\lambda}}$ is nonnegative and nonzero.

\subsection{Completion of the higher-dimensional construction}
In this subsection, we pass the construction from the plane to the higher dimension,

\medskip
\paragraph{\it Fixing the solved disk data}
We now pass from the weighted disk equation to an axially symmetric domain in
$\mathbb R^n$.  We keep the parameters $a$, $\tau$, and $\lambda$ fixed as
chosen above, and we write
$$
\delta_*:=\delta_{\tau,\lambda},
\qquad W_*:=W(a,\delta_*,\tau,\lambda).
$$
Let $F:=F_{W_*,a,\delta_*}$ be the conformal map defined  via \eqref{eq:hd-F-prime-corrected} as
$$
F'(z)=\exp\{-\mathcal H[W_*\,dm+\mu_{a,\delta_*}](z)\}, \qquad F(0)=0.
$$
Recall that
$$
A_{\tau,F}(z):=(1+\tau\,\operatorname{Im}F(z))^p,
$$
and also denote by
$  u_{\mathbb D}:=v_{W_*,a,\delta_*,\tau,\lambda}$
the disk-side semilinear solution obtained in the previous subsection. In particular, according to \eqref{eq:hd-disk-semilinear-corrected} 
$u_{\mathbb D}$ solves
$$
-\operatorname{div}(A_{\tau,F}\nabla u_{\mathbb D})
=A_{\tau,F}|F'|^2 f(\lambda^2u_{\mathbb D})    \quad\text{in }\mathbb D,    
\qquad u_{\mathbb D}=0\quad\text{on }\mathbb T .
$$

\medskip
\paragraph{\it The equation on the meridian image}
Set
$\Omega^{\rm mer}:=F(\mathbb D)\subset\mathbb R^2.$
We write a point of $\mathbb R^2$ as $(x,y)$.  Define the corresponding solution on the meridian image by
$$
 u_{\rm mer}(F(z)):=u_{\mathbb D}(z), \qquad z\in\mathbb D .
$$
Since $F$ is conformal, the equation for $u_{\mathbb D}$ is equivalent to
\begin{equation}\label{eq:umer} 
  -\operatorname{div}_{x,y}
 \left((1+\tau y)^p\nabla u_{\rm mer}\right)
 = (1+\tau y)^p f(\lambda^2u_{\rm mer}) \quad\text{in }\Omega^{\rm mer},   
\end{equation}
with zero boundary trace on $\partial\Omega^{\rm mer}$.

\medskip
\paragraph{\it The weighted quadrature identity on the meridian image}
The boundary equation \eqref{eq:hd-weighted-boundary-equation-corrected} now
gives a weighted quadrature identity on $\Omega^{\rm mer}$.  Indeed, let
$V$ solve
$$
    \operatorname{div}_{x,y}\left((1+\tau y)^p\nabla V\right)=0
$$
in $\Omega^{\rm mer}$, with boundary trace in the disk trace class.  Then
$V\circ F$ is $A_{\tau,F}$-harmonic in $\mathbb D$.  Then using the definition of $\mathcal T_{\tau,F}$ and \eqref{eq:hd-weighted-boundary-equation-corrected}, we
obtain
$$
 \int_{\mathbb D} (V\circ F)\, A_{\tau,F}|F'|^2 f(\lambda^2u_{\mathbb D})\,da = C_{\tau,\lambda} \int_{\mathbb T}
(V\circ F)^* (A_{\tau,F})^*|(F^{\prime })^*|\,dm .
$$
Changing variables by $F$, and recalling the normalizations
$da=dA/\pi$ and $ds=2\pi |F^{\prime *}|\,dm$, this becomes
\begin{equation}\label{eq:weighted-identity}
\int_{\Omega^{\rm mer}} V(x,y)(1+\tau y)^p f(\lambda^2u_{\rm mer}(x,y))\,dx\,dy
=c_{\rm mer} \int_{\partial\Omega^{\rm mer}} V(x,y)(1+\tau y)^p\,ds,
\end{equation}
where
$ c_{\rm mer}:=\frac{C_{\tau,\lambda}}{2}>0.$

\medskip
\paragraph{\it Translation away from the axis}
We now translate the meridian domain far from the axis. Set $ R:=\tau^{-1}$
and define
$$
D:=\{(x,r):(x,r-R)\in\Omega^{\rm mer}\}.
$$
Since $\tau>0$ is small and $\Omega^{\rm mer}$ is bounded, we have
$$
    D\subset\subset\{(x,r):r>0\}.
$$
Define the translated meridian solution
$$
 u_D(x,r):=u_{\rm mer}(x,r-R), \qquad (x,r)\in D .
$$
If $y=r-R$, then
$ r^p=R^p(1+\tau y)^p.$
Consequently, \eqref{eq:umer} becomes
$$
-\operatorname{div}_{x,r}(r^p\nabla u_D)
= r^p f(\lambda^2u_D) \quad\text{in }D.
$$
Equivalently,
$$
-L_pu_D=f(\lambda^2u_D) \quad\text{in }D,\qquad u_D=0\quad\text{on }\partial D,
$$
where
$$
    L_p=\partial_{xx}+\partial_{rr}+\frac{p}{r}\partial_r.
$$

The weighted quadrature identity \eqref{eq:weighted-identity} also translates.
Let $\psi$ be any admissible variational $L_p$-harmonic replacement in $D$ with
boundary trace in the corresponding Sobolev trace class, and set
$V(x,y):=\psi(x,R+y)$.  Then
$$
\operatorname{div}_{x,y}\left((1+\tau y)^p\nabla V\right)=0
\quad\text{in }\Omega^{\rm mer}.
$$
Applying \eqref{eq:weighted-identity} to this replacement and multiplying both
sides by $R^p$, we obtain
\begin{equation}\label{eq:hd-meridian-moment-variational}
 \int_D f(\lambda^2u_D)\psi\,r^p\,dx\,dr
=c_{\rm mer}\int_{\partial D}\psi\,r^p\,ds .
\end{equation}
Thus \eqref{eq:meridian-weighted-moment} holds with
$q=f(\lambda^2u_D)$ and $c=c_{\rm mer}$ for all variational $L_p$-harmonic
replacements in the required trace class.

\medskip
\paragraph{\it Rotation to an axially symmetric domain}
We next rotate $D$ around the $x$-axis.  Define
$$
\Omega^{\rm ref} := \{(x,z)\in\mathbb R\times\mathbb R^{n-1}:(x,|z|)\in D\}.
$$
Define the reference axially symmetric solution
$$
U^{\rm ref}(x,z):=u_D(x,|z|), \qquad (x,z)\in\Omega^{\rm ref}.
$$
Since $D\subset\subset\{r>0\}$, the rotation is regular.  The equation
$-L_pu_D=f(\lambda^2u_D)$ becomes
$$
-\Delta U^{\rm ref}
=f(\lambda^2U^{\rm ref}) \quad\text{in }\Omega^{\rm ref},
$$
with zero Sobolev trace on $\partial\Omega^{\rm ref}$.

By Lemma~\ref{lem:rotated-moment-to-weak-law}, applied to
$q=f(\lambda^2u_D)$ and the moment identity
\eqref{eq:hd-meridian-moment-variational}, the zero extension of
$U^{\rm ref}$ satisfies
\begin{equation}\label{eq:hd-ref-weak-law}
\Delta U^{\rm ref}
=c_{\rm mer}\,\mathcal H^{n-1}\!\lfloor_{\partial^*\Omega^{\rm ref}}
-f(\lambda^2U^{\rm ref})\mathbf 1_{\Omega^{\rm ref}}\,dX
\qquad\text{in }\mathcal D'(\mathbb R^n).
\end{equation}

\medskip
\paragraph{\it Final scaling and the distributional Bernoulli law}
Finally, we scale to the original domain.  Define
$$
    \Omega:=\lambda\Omega^{\rm ref},
$$
and define $U$ on $\Omega$ by
$$
    U(\lambda X):=\lambda^2U^{\rm ref}(X),
    \qquad X\in\Omega^{\rm ref}.
$$
Then
$$
    -\Delta U=f(U)\quad\text{in }\Omega,
$$
and $U$ has zero Sobolev trace on $\partial\Omega$.  Testing
\eqref{eq:hd-ref-weak-law} against $\varphi(\lambda\cdot)$ gives, for the zero
extension of $U$,
\begin{equation}\label{eq:hd-final-weak-law}
\Delta U
=c\,\mathcal H^{n-1}\!\lfloor_{\partial^*\Omega}
-f(U)\mathbf1_\Omega\,dX
\qquad\text{in }\mathcal D'(\mathbb R^n),
\end{equation}
where
$$
    c:=\lambda c_{\rm mer}>0.
$$
The zero extension belongs to $W^{1,2}(\mathbb R^n)$ since $U$ has zero Sobolev
trace on $\partial\Omega$.

\medskip


\paragraph{\it Topology and non-sphericity}
The singular measure $\mu_{a,\delta_*}$ is nonzero, so the meridian domain is
non-Smirnov.  In particular $D$ is not a disk.  Since
$D\subset\subset\{r>0\}$ is a Jordan domain, the rotated domain has product
topology
$$
    \Omega^{\rm ref}\simeq \mathbb B^2\times\mathbb S^{n-2},
$$
and its boundary has product topology
$$
    \partial\Omega^{\rm ref}\simeq \mathbb S^1\times\mathbb S^{n-2}.
$$
The same is true after the scaling by $\lambda$.  In particular, for $n\ge3$,
$\Omega$ is not homeomorphic to a ball, and hence is not a ball.

\medskip
\paragraph{\it Finite perimeter and failure of upper density}
Finally, we check the estimates on the boundary. Let
$\Gamma:=\partial D .$ Since $D$ is a rectifiable Jordan domain, $D$ is a set of finite perimeter in
the plane and
$$
\mathcal H^1(\Gamma)<\infty, \qquad \mathcal H^1(\Gamma\setminus\partial^*D)=0.
$$
The second property follows from the standard structure theorem for rectifiable Jordan domains: At $\mathcal H^1$-a.e. point of a rectifiable Jordan curve there is an approximate tangent, and the blow-up of the Jordan domain is a half-plane; see \cite[Chapter 1 \& 10]{P1992}.

Since $D\subset\subset \{r>0\}$, there exist constants $0<r_-<r_+<\infty$ such that
$$
  r_-\le r\le r_+  \qquad\text{for every }(x,r)\in\Gamma .
$$
Parametrize $\Gamma$ by arclength,
$$
  \gamma(s)=(x(s),r(s)), \qquad s\in[0,\mathcal H^1(\Gamma)].
$$
The rotated boundary is parametrized, up to a negligible set, by
$$
 \Phi(s,\omega):=(x(s),r(s)\omega), \qquad \omega\in\mathbb S^{n-2}.
$$
The tangential derivative in the $s$-direction is
$$
\partial_s\Phi=(x'(s),r'(s)\omega),
$$
and the derivatives in the spherical directions are multiplied by the factor $r(s)$.
Since $|\gamma'(s)|=1$, the area Jacobian is
$$
 J_\Phi(s,\omega)=r(s)^{\,n-2}=r(s)^p .
$$
Thus the area formula gives
$$
  \mathcal H^{n-1}(\partial^*\Omega)  = \lambda^{n-1} \mathcal H^{n-2}(\mathbb S^{n-2}) \int_{\partial^*D} r^p\,d\mathcal H^1;
$$
recall that $\Omega:=\lambda\,\Omega^{\rm ref}$. 
Using $r_-\le r\le r_+$ and $0<\mathcal H^1(\Gamma)<\infty$, we obtain
$$
0<\mathcal H^{n-1}(\partial^*\Omega)<\infty .
$$
Moreover,
$$
\partial\Omega
=
\{\lambda(x,r\omega):(x,r)\in\Gamma,\ \omega\in\mathbb S^{n-2}\},
$$
and the same area formula gives
$$
\mathcal H^{n-1}(\partial\Omega\setminus\partial^*\Omega)
\le \lambda^{n-1}\mathcal H^{n-2}(\mathbb S^{n-2})
\int_{\Gamma\setminus\partial^*D} r^p\,d\mathcal H^1=0.
$$
This proves the first two claims.

It remains to prove the failure of the uniform upper density bound.  Suppose, toward
a contradiction, that
$$
 M:={\rm ess}\sup_{X\in\partial^*\Omega}
\sup_{0<\rho<1} \frac{\mathcal H^{n-1}(B_\rho(X)\cap\partial^*\Omega)}{\rho^{n-1}} <\infty.
$$
Then for $\mathcal H^{n-1}$-a.e. $X_0\in\partial^*\Omega$,
\begin{equation}\label{eq:all-center-upper-density}
\mathcal H^{n-1}(B_\rho(X_0)\cap\partial^*\Omega)
\le CM\rho^{n-1}
\end{equation}
for every $X_0\in\partial\Omega$ and all sufficiently small $\rho>0$.
This immediately yields a contradiction to the main result of \cite{FZ2025} when $f\equiv 1$. We next verify this inconsistency  directly according to our construction.

We first show that this implies a uniform upper linear density bound for the meridian
curve $\Gamma$.  Fix $y_0=(x_0,r_0)\in\Gamma$ and $\omega_0\in\mathbb S^{n-2},$
and set
$$
    X_0:=\lambda(x_0,r_0\omega_0)\in\partial\Omega .
$$
Let $0<s<s_0$, where $s_0>0$ is chosen so small that $s_0<r_-/10$, and put $I_s:=\Gamma\cap B_s(y_0).$
Let
$$
    \mathcal C_s:=\{\omega\in\mathbb S^{n-2}:|\omega-\omega_0|<c_0s\},
$$
with $c_0>0$ chosen depending only on $r_+$.  For every
$(x,r)\in I_s$ and every $\omega\in\mathcal C_s$,
$$
\lambda(x,r\omega)\in B_{C\lambda s}(X_0),
$$
where $C$ depends only on $r_+$.  Hence
$$
\{\lambda(x,r\omega):(x,r)\in I_s,\ \omega\in\mathcal C_s\}
\subset B_{C\lambda s}(X_0)\cap\partial\Omega .
$$
The area formula on this patch and the lower bound $r\ge r_-$ give
$$
\mathcal H^{n-1}(B_{C\lambda s}(X_0)\cap\partial^*\Omega)
\ge c\lambda^{n-1}s^{n-2}\mathcal H^1(I_s),
$$
where $c>0$ depends only on $n,r_-,r_+$.  Combining this with
\eqref{eq:all-center-upper-density} at scale $C\lambda s$, we get
$$
    c\lambda^{n-1}s^{n-2}\mathcal H^1(\Gamma\cap B_s(y_0))
    \le CM(\lambda s)^{n-1}.
$$
After cancellation,
$$
    \mathcal H^1(\Gamma\cap B_s(y_0))\le CMs.
$$
For larger $s$, the same estimate follows after increasing $C$, since
$\mathcal H^1(\Gamma)<\infty$.  Therefore
$$
    \sup_{y\in\Gamma}\sup_{s>0}
    \frac{\mathcal H^1(\Gamma\cap B_s(y))}{s}<\infty.
$$
Thus $\Gamma$ is a rectifiable Jordan curve with a uniform upper linear density
bound, i.e. it is Ahlfors regular.  By Zinsmeister's result \cite{Z1984}, recorded for instance as \cite[Theorem~7.6]{P1992}, the meridian domain is Smirnov.  This contradicts the construction:
The conformal derivative of the meridian map has a nontrivial singular inner
factor since the singular measure $\mu_{a,\delta_{\tau,\lambda}}$ is nonzero.  Hence
$$
{\rm ess}\sup_{X\in\partial^*\Omega}
 \sup_{0<\rho<1} \frac{\mathcal H^{n-1}(B_\rho(X)\cap\partial^*\Omega)}{\rho^{n-1}}
 =\infty .
$$
This concludes the proof of Theorem~\ref{thm:higher-dimensional-counterexample}.

\medskip

\noindent{\bf  Data availability}: Data sharing is not applicable to this article since no datasets were generated or analysed in the course of the present study.

\end{document}